\documentclass[12pt, a4paper, reqno]{amsart}
\usepackage{amsmath}
\usepackage{amsfonts}
\usepackage{amssymb}
\usepackage{amsthm}
\usepackage{url}
\usepackage{color}
\usepackage{array}

\newcommand{\R}{\mathbb{R}}

\newcommand{\N}{\mathbb{N}}
\newcommand{\Q}{\mathbb{Q}}

\DeclareMathOperator*{\osc}{osc}

\DeclareMathOperator*{\diam}{diam}
\DeclareMathOperator*{\spt}{spt}
\DeclareMathOperator*{\dist}{dist} 
\DeclareMathOperator*{\Lip}{Lip}

\DeclareMathOperator{\Mod}{Mod}
\DeclareMathOperator{\capa}{Cap}

\DeclareMathOperator*{\supp}{supp}
\DeclareMathOperator*{\Arg}{Arg}

\def\vint_#1{\mathchoice%
          {\mathop{\kern 0.2em\vrule width 0.6em height 0.69678ex depth -0.58065ex
                  \kern -0.8em \intop}\nolimits_{\kern -0.4em#1}}%
          {\mathop{\kern 0.1em\vrule width 0.5em height 0.69678ex depth -0.60387ex
                  \kern -0.6em \intop}\nolimits_{#1}}%
          {\mathop{\kern 0.1em\vrule width 0.5em height 0.69678ex depth -0.60387ex
                  \kern -0.6em \intop}\nolimits_{#1}}%
          {\mathop{\kern 0.1em\vrule width 0.5em height 0.69678ex depth -0.60387ex
                  \kern -0.6em \intop}\nolimits_{#1}}}

\newcommand{\p}{{$p\mspace{1mu}$}}

\newcommand{\Om}{\Omega}
\newcommand{\loc}{\mathrm{loc}}

\newcommand{\eps}{\varepsilon}

\newcommand{\BV}{\mathrm{BV}}
\newcommand{\liploc}{\mathrm{Lip}_{\mathrm{loc}}}
\newcommand{\ch}{\text{\raise 1.3pt \hbox{$\chi$}\kern-0.2pt}}

\theoremstyle{plain}
\newtheorem{theorem}[equation]{Theorem}
\newtheorem{lemma}[equation]{Lemma}
\newtheorem{corollary}[equation]{Corollary}
\newtheorem{proposition}[equation]{Proposition}

\numberwithin{equation}{section}

\theoremstyle{definition}
\newtheorem{definition}[equation]{Definition}

\newtheorem{example}[equation]{Example}

\theoremstyle{remark}
\newtheorem{remark}[equation]{Remark}

\begin{document}

\title[Dirichlet problem for least gradient]{Notions of Dirichlet problem\\
for functions of least gradient\\
in metric measure spaces}

\author[Korte, Lahti, Li, Shanmugalingam]{Riikka Korte, Panu Lahti, Xining Li, Nageswari Shanmugalingam}

\date{\today}



\maketitle

\begin{abstract}
We study two notions of Dirichlet problem associated with BV energy minimizers
(also called functions of least
gradient) in bounded domains in metric measure spaces whose measure is doubling and supports a 
$(1,1)$-Poincar\'e inequality. Since one of the two notions is not amenable to the direct method of the calculus of
variations, we construct, based on an approach of~\cite{Juu, MRL}, solutions
by considering the Dirichlet problem for $p$-harmonic functions, $p>1$, and letting $p\to 1$. Tools developed and
used in this paper include the \emph{inner} perimeter measure of a domain. 
\end{abstract}

\tableofcontents

\newpage

\section{Introduction}

Existence, uniqueness, continuity, and stability of solutions to the Dirichlet problem for $p$-harmonic
functions in metric measure space
setting is now reasonably well understood when $1<p<\infty$. The corresponding problem for $p=1$, that is, finding a BV
function of least gradient in the given domain, with prescribed trace on the boundary, is not well understood.  Part
of the problem is that without additional curvature restrictions for the boundary of the given domain, solutions to
the Dirichlet problem, where the \emph{trace} of the BV function is prescribed, are known to not always exist. Thus
alternate notions of Dirichlet problem for the least gradient functions need to be explored. Based on the notion of
Dirichlet problem set forth in~\cite{Giu84}, in~\cite{HKL} a notion of Dirichlet problem was proposed 
(\cite{HKL} considers the area functional, but the results are easily applicable to the total variation functional). 
It was shown in~\cite{HKL} that for a wide class of domains in metric measure spaces equipped with a doubling
measure supporting a $(1,1)$-Poincar\'e inequality, solutions always exist if the boundary data are themselves given by
a BV function. The notion  proposed there required extension of the BV solution to the exterior of the domain of the problem.

In this paper we discuss an alternate notion of the Dirichlet problem for least gradient functions that does not require
extension of the BV solution to the complement of the domain of interest. The boundary data is given by a fixed Lipschitz function.
However, unlike in~\cite{HKL}, the 
direct method of the calculus of variations does not yield existence of solutions for this notion of the Dirichlet problem.
Thus an alternate method of verifying existence needs to be adopted. In~\cite[Theorem~3.1]{Juu} it was
shown, using the tools of viscosity solutions, that the limit of a sequence of $p$-harmonic functions in a Euclidean
domain, as $p\to 1$, must be a function of least gradient. In the recent paper~\cite{MRL} it was shown that such a
limit function, again in the Euclidean setting, satisfies the notion of Dirichlet problem considered in this paper.
The key tool used in~\cite{MRL} is the divergence theorem. In our setting of metric measure spaces we do not have
access to the divergence theorem nor notions of viscosity solutions. We instead employ a careful study of 
\emph{inner trace} of BV functions for a class of domains.  

We start by showing that if there is a sequence $u_{p_k}$ of $p_k$-harmonic functions with $(p_k)_k$ a monotone
decreasing sequence of real numbers larger than $1$ such that $\lim_kp_k=1$, and $u_{p_k}$ converges 
to $u$ in $L^1$, then the limit function $u$ is a function of least gradient, see Theorem~\ref{thm:main1}.
In the  case of $p$-energy with $p>1$, there is no ambiguity in the sense in which we want to fix the boundary values of the 
function, if the boundary values are themselves restrictions of Sobolev functions. Note that Lipschitz functions are
a priori in the Sobolev class $N^{1,p}$ for each $1\le p<\infty$.
However, when $p=1$ and the solutions are merely functions of bounded variation, it is not 
clear what notion of the Dirichlet problem is the correct one. 

In this paper, we propose two ways of defining solutions to the Dirichlet problem: the first one,
described in Definition~\ref{def:dirichlet}(B), is based on 
minimizing the BV-energy in the closure of the domain. In the second one, 
given in Definition~\ref{def:dirichlet}(T), extension of solutions to the complement of the domain is not required, but the energy  being minimized 
includes the integral of the jump in the inner trace of the BV function (in comparison with the boundary data) 
measured with respect to the interior perimeter of the domain.

The drawback of the first  approach 
is that the structure of the underlying space close to the boundary but outside the domain also affects the minimization 
problem. This phenomenon occurs already in weighted Euclidean spaces; see the discussion following 
Definition~\ref{def:dirichlet}. On the other hand, the advantage of the first approach is that 
the energy being minimized is lower semicontinuous with respect to $L^1$-convergence, and hence
existence of solutions can be proven using the direct method of the calculus of variations. 
In the Euclidean setting, Dirichlet problems 
related to minimizing convex functionals with linear growth have been studied in~\cite{BS}, and
the notion of Dirichlet problem considered there is also equivalent to the notion given by Definition~\ref{def:dirichlet}(B) here.
The second approach given in Definition~\ref{def:dirichlet}(T) avoids 
the impact of the part of the complement of the domain that is near the boundary of the domain, but the drawback is
that proving the existence of solutions using the direct method of the calculus of variations is not possible. In the setting
of metric measure spaces considered here, we do not even have the tools of divergence or Green's theorem, and hence 
our proof is more involved. 

One benefit of the proof we provide here is that the results hold even in a wider class of Euclidean domains; the standard
theory from~\cite{MRL} only consider smooth domains, while~\cite{BS} considers Euclidean Lipschitz domains.

%



The structure of this paper is as follows. In Section~2 we explain the notation and definitions of concepts used
in this paper. In Section~3 we show that functions that arise as $L^1$-limits of $p$-harmonic functions are
functions of least gradient, see Theorem~\ref{thm:main1}. The focus of the fourth section is to describe the two
notions of solution to the Dirichlet problem, see Definition~\ref{def:dirichlet}, while the fifth section gives a way
of finding good Lipschitz approximations of BV functions via discrete convolutions. Such discrete convolutions are used
in Section~6 to compare the inner perimeter measure $P_+(\Om,\cdot)$ of the bounded domain $\Om$ with its
perimeter measure $P(\Om, \cdot)$, see Theorem~\ref{thm:P+vsP}. 

In Section~7, we show that the least gradient functions, obtained as $L^1$-limits of $p$-harmonic functions that are solutions to the Dirichlet
problem with the fixed Lipschitz boundary data, are necessarily solutions to the Dirichlet problem
defined in Definition~\ref{def:dirichlet}(T) with the same Lipschitz boundary data. This result is Theorem~\ref{thm:dirichletT}. For this result,  we need some additional assumptions on $\Om$. More precisely, we need to assume that $\Om$ is of finite perimeter and that at $\mathcal{H}$-a.e.~boundary
point of $\Om$ the complement of $\Om$ has positive density.


The focus of Section~8 is to show that in addition to perturbing the BV energy to the $L^p$-energy (via
$p$-harmonic functions), if we also perturb the domain by approximating the domain from outside, then
the corresponding $p$-harmonic solutions have a subsequence that converges to a solution to the Dirichlet problem
as given in Definition~\ref{def:dirichlet}(B). While the problem~(T) is associated with approximating the domain
from inside, the results of Section~8 show that the problem~(B) is associated with approximating the domain
from outside; see Theorem~\ref{thm:ProblemB}. It should be noted that the restrictions placed on the domain in
relation to problem~(T) as in Section~7 are not needed in Section~8.
Finally, in Section~9 we consider alternate notions of
functions of least gradient, and show that all these notions coincide. For the convenience of the reader,
in the appendix we provide a proof of the fact that the inner perimeter measure $P_+(\Om,\cdot)$ as considered in 
Definition~\ref{def:innerPerim} is indeed a Radon measure.

\medskip

{\bf Acknowledgements.} The research of N.S. is partially supported by the grant \#~DMS--1500440 of NSF (U.S.A.). 
P.L. was supported by a grant from the Finnish Cultural Foundation.
Part of the research was conducted during the visit of N.S. to Aalto University, and during the visit of X.L.  to 
University of Cincinnati.  Some parts of the research was conducted during the time spent by X.L. as a postdoctoral 
scholar at Aalto University. The authors wish to thank these institutions for their kind hospitality. The authors also thank
Juha Kinnunen for making them aware of the reference~\cite{MRL} and for fruitful discussions on the topic.

\section{Preliminaries}

Throughout this paper we assume that $(X,d,\mu)$ is a complete metric space equipped
with a Borel regular outer measure $\mu$ that satisfies a doubling property and supports a $(1,1)$-Poincar\'e inequality (see definitions below). 
We assume that $X$ consists of at least 2 points. The doubling property means that there exists a constant $C_d\ge 1$ such that
\[
0<\mu(B(x,2r))\leq C_d\mu(B(x,r))<\infty
\]
for every ball $B(x,r)\subset X$.  Given a ball $B=B(x,r)$ and $\tau>0$, we denote by $\tau B$  the ball $B(x,\tau r)$. 
In a metric space, a ball does not necessarily have a unique center 
and radius, but whenever we use the above abbreviation we will consider balls whose center and radii have been 
pre-specified.

In general, $C\ge 1$ will denote a generic constant whose particular value is not important for the purposes of this
paper, and might differ between
each occurrence. When we want to specify that a constant $C$
depends on the parameters $a,b, \ldots,$ we write $C=C(a,b,\ldots)$. Unless otherwise specified, all constants only 
depend on the doubling constant $C_d$ and the constants $C_P,\lambda$ associated
with the Poincar\'e inequality defined below.

A complete metric space with a doubling measure is proper,
that is, closed and bounded sets are compact. Since $X$ is proper, for any open set $\Omega\subset X$
we define $\liploc(\Omega)$ to be the space of
functions that are Lipschitz in every $\Omega'\Subset\Omega$.
Here $\Omega'\Subset\Omega$ means that $\Omega'$ is open and that $\overline{\Omega'}$ is a
compact subset of $\Omega$.
We define other local spaces similarly.

For any set $A\subset X$, and $0<R<\infty$, the restricted spherical Hausdorff content of codimension $1$ is defined by
\[
\mathcal H_{R}(A)=\inf\left\{ \sum_{i=1}^\infty\frac{\mu(B(x_{i},r_{i}))}{r_{i}} :\, A\subset\bigcup_{i=1}^\infty B(x_{i},r_{i}), \,r_{i}\leq R  \right\}.
\]
The codimension $1$ Hausdorff measure of a set $A\subset X$ is 
\[
\mathcal H(A)=\lim_{R\rightarrow 0}\mathcal H_{R}(A).
\]
The codimension $1$ Minkowski content of a set $A\subset X$ is defined for any positive
Radon measure $\nu$ by
\begin{equation}\label{eq:definition of Minkowski content}
\nu^{+}(A):=\liminf_{R\to 0}\frac{\nu\left(\bigcup_{x\in A}B(x,R)\right)}{2R}.
\end{equation}

\begin{definition}\label{def:int-ext-bdy-meas}
The measure theoretic boundary $\partial^{*}E$ of a set $E\subset X$ is the
set of
all points $x\in X$
at which both $E$ and its complement have positive upper density, i.e.
\[
\limsup_{r\to 0^+}\frac{\mu(B(x,r)\cap E)}{\mu(B(x,r))}>0\quad\;
  \textrm{and}\quad\;\limsup_{r\to 0^+}\frac{\mu(B(x,r)\setminus E)}{\mu(B(x,r))}>0.
\]
The measure theoretic interior $I_E$ is the set of all points $x\in X$ for which
\[
\lim_{r\to0^+}\frac{\mu(B(x,r)\setminus E)}{\mu(B(x,r))}=0,
\]
and the measure theoretic exterior $O_E$ is the set of all points $x\in X$ for which
\[
\lim_{r\to0^+}\frac{\mu(B(x,r)\cap E)}{\mu(B(x,r))}=0.
\]
\end{definition}
Observe that 
$\partial^*E=X\setminus(I_E\cup O_E)$.
Note that when $E$ is open, $E\subset I_E$. 
See the discussion following~\eqref{eq:density of E} for more on the
relationship between the measure theoretic boundary and the perimeter measure.

A curve is a rectifiable continuous mapping from a compact interval
into $X$.

\begin{definition}
A nonnegative Borel function $g$ on $X$ is an \emph{upper gradient} 
of an extended real-valued function $u$
on $X$ if for all curves $\gamma$ on $X$, we have
\begin{equation}\label{eq:upper gradient definition}
|u(x)-u(y)|\le \int_\gamma g\,ds,
\end{equation}
where $x$ and $y$ are the end points of $\gamma$. We interpret $|u(x)-u(y)|=\infty$ whenever  
at least one of $|u(x)|$, $|u(y)|$ is infinite. 
\end{definition}

By replacing $X$ with a set $A\subset X$ and considering curves $\gamma$ in $A$, we can talk about a function $g$ being an upper gradient of $u$ in $A$.
Upper gradients were originally introduced in~\cite{HK}.

We define the local Lipschitz constant of a locally Lipschitz function $u\in\liploc(X)$ by
\begin{equation}\label{eq:local Lipschitz constant}
\Lip u(x):=\limsup_{r\to 0^+}\sup_{y\in B(x,r)\setminus \{x\}}\frac{|u(y)-u(x)|}{d(y,x)}.
\end{equation}
Then $\Lip u$ is an upper gradient of $u$, see e.g.~\cite[Proposition 1.11]{Che}.

It is easy to check that if $u,v\in \liploc(X)$ and $\alpha,\beta\ge 0$, then we have the subadditivity
\begin{equation}\label{eq:subadditivity of local Lipschitz constants}
\Lip (\alpha u+\beta v)(x)\le \alpha \Lip u(x)+\beta \Lip v(x)\quad \textrm{for }\textrm{every }x\in X.
\end{equation}

Let $\Gamma$ be a family of curves,
and let $1\le p<\infty$. The $p$-modulus of $\Gamma$ is defined
by
\[
\Mod_{p}(\Gamma):=\inf\int_{X}\rho^{p}\, d\mu
\]
where the infimum is taken over all nonnegative Borel functions $\rho$
such that $\int_{\gamma}\rho\,ds\ge 1$ for every $\gamma\in\Gamma$.
If a property fails only for a curve family with $p$-modulus zero,
we say that it holds for $p$-almost every (a.e.) curve.

\begin{definition}
If $g$ is a nonnegative $\mu$-measurable function on
$X$ and~\eqref{eq:upper gradient definition} holds for $p$-almost every curve, then
$g$ is a $p$\emph{-weak upper gradient} of $u$.
It is known that if $u$ has an upper gradient $g\in L^p_{\loc}(\Om)$
in $\Om$, then
there exists a minimal $p$-weak
upper gradient of $u$ in $\Omega$, which we always denote by $g_{u}$, satisfying $g_{u}(x)\le g(x)$ for
$\mu$-a.e. $x\in \Omega$, for any $p$-weak upper gradient $g\in L_{\loc}^{p}(\Omega)$
of $u$ in $\Omega$,
see~\cite[Theorem 2.25]{BB}.
\end{definition}

\begin{remark}\label{rmk:minimal weak upper gradients}
Note that a priori the minimal $p$-weak upper gradient $g_u$ of $u$ may depend on $p$.
However, if $u$ has a minimal $q$-weak upper gradient $g_0$ in $\Omega$ with $1\le q<p$, then
$g_0\le g_u$ $\mu$-a.e.~in $\Omega$ because a $p$-weak upper gradient of $u$ is automatically a
$q$-weak upper gradient of $u$.
Also, a minimal $p$-weak upper gradient in $\Omega$ is also a minimal $p$-weak
upper gradient in any open $U\subset \Omega$.

From the results in~\cite{Che} (see~\cite{HKST} for further exposition on this) it follows that when 
the measure $\mu$ on $X$ is doubling and supports a $(1,1)$-Poincar\'e inequality, the
minimal $p$-weak upper gradient of a locally Lipschitz function $u$ on $\Omega$ is
$\Lip u$ for all $1<p<\infty$.
\end{remark}

We consider the following norm
\[
\Vert u\Vert_{N^{1,p}(X)}:=\Vert u\Vert_{L^p(X)}+\inf_g\Vert g\Vert_{L^p(X)},
\]
with the infimum taken over all upper gradients $g$ of $u$. 

\begin{definition}
The substitute for the Sobolev space $W^{1,p}(\R^n)$ in the metric setting is the following 
Newton-Sobolev space 
\[
N^{1,p}(X):=\{u:\|u\|_{N^{1,p}(X)}<\infty\}/{\sim},
\]
where the equivalence relation $\sim$ is given by $u\sim v$ if and only if 
\[
\Vert u-v\Vert_{N^{1,p}(X)}=0.
\]
Similarly, we can define $N^{1,p}(\Omega)$ for any open set $\Omega\subset X$.
For more on Newton-Sobolev spaces, we refer to~\cite{S, HKST, BB}.
\end{definition}

The $p$-capacity of a set $A\subset X$ is given by
\[
 \capa_p(A):=\inf \Vert u\Vert_{N^{1,p}(X)},
\]
where the infimum is taken over all functions $u\in N^{1,p}(X)$ such that $u\ge 1$ in $A$.

\begin{remark}
When $\mu$ is doubling and supports a $(1,p)$-Poincar\'e 
inequality, then Lipschitz functions are dense in $N^{1,p}(X)$. When $X$ is complete and
$\mu$ is doubling, even if $X$ does not support a $(1,p)$-Poincar\'e inequality 
Lipschitz functions are still dense in $N^{1,p}(X)$; this follows from the deep results
in~\cite{AGS}.
\end{remark}

Next we recall the definition and basic properties of functions
of bounded variation on metric spaces, see~\cite{Miranda03}. See also
e.g.~\cite{AFP, Fed, Giu84, Zie89} for the classical 
theory in the Euclidean setting.
For $u\in L^1_{\loc}(X)$, we define the total variation of $u$ on $X$ to be 
\[
\|Du\|(X):=\inf\Big\{\liminf_{i\to\infty}\int_X g_{u_i}\,d\mu: u_i\in \liploc(X), u_i\to u\textrm{ in } L^1_{\loc}(X)\Big\},
\]
where each $g_{u_i}$ is the minimal $1$-weak upper gradient of $u_i$.
\label{p-weak upper gradients in total variation}
Note that instead of merely requiring
$u_i\to u$ in $L^1_{\loc}(X)$ we could require $u_i-u\to 0$ in $L^1(X)$. It turns out that even with this 
stricter definition, the norm $\Vert Du\Vert(X)$ does not change;
see Lemma~\ref{lem:L1 loc and L1 convergence}.
Note also that by~\cite[Theorem 1.1]{AmDi} and Remark~\ref{rmk:minimal weak upper gradients}, we can replace
$g_{u_i}$ by the minimal $p$-weak upper gradient $\Lip u_i$, for $p>1$.

We say that a function $u\in L^1(X)$ is \emph{of bounded variation}, 
and denote $u\in \BV(X)$, if $\|Du\|(X)<\infty$. 
A $\mu$-measurable set $E\subset X$ is said to be of \emph{finite perimeter} if $\|D\ch_E\|(X)<\infty$. The perimeter
of $E$ in $X$ is also denoted by
\[
P(E,X):=\|D\ch_E\|(X).
\]
By replacing $X$ with an open set $U\subset X$ in the definition of the total variation, we can define $\|Du\|(U)$.
The $\BV$ norm is given by
\begin{equation*}
\Vert u\Vert_{\BV(U)}:=\Vert u\Vert_{L^1(U)}+\Vert Du\Vert(U).
\end{equation*}
It was shown in~\cite[Theorem~3.4]{Miranda03} that for $u\in \BV(X)$, $\Vert Du\Vert$ is the restriction to the class of 
open sets of a finite Radon measure defined on the
class of all subsets of $X$. This outer measure is obtained from the map $U\mapsto\Vert Du\Vert(U)$ on open sets
$U\subset X$ via the standard Carath\'eodory construction. Thus, 
for an arbitrary set $A\subset X$, 
\[
\|Du\|(A):=\inf\bigl\{\|Du\|(U):\,U\text{ open},\, A\subset U \bigr\}.
\]
Similarly, if $u\in L^1_{\loc}(U)$ with $\Vert Du\Vert(U)<\infty$,
then $\|Du\|(\cdot)$ is a finite Radon measure on $U$.

For any Borel sets $E_1,E_2\subset X$, we have by~\cite[Proposition 4.7]{Miranda03}
\[
P(E_1\cap E_2,X)+P(E_1\cup E_2,X)\le P(E_1,X)+P(E_2,X). 
\]
The proof works equally well for $\mu$-measurable $E_1,E_2\subset X$ and with $X$
replaced by any open set, and then by approximating an arbitrary set $A\subset X$
from the outside by open sets we obtain
\begin{equation}\label{eq:Caccioppoli sets form an algebra}
P(E_1\cap E_2,A)+P(E_1\cup E_2,A)\le P(E_1,A)+P(E_2,A). 
\end{equation}

We have the following coarea formula from~\cite[Proposition 4.2]{Miranda03}: if $F\subset X$ is a Borel set and 
$u\in \BV(X)$, then
\begin{equation}\label{eq:coarea}
\|Du\|(F)=\int_{-\infty}^{\infty}P(\{u>t\},F)\,dt.
\end{equation}
In particular, the map $t\mapsto P(\{u>t\},F)$ is Lebesgue measurable on $\R$.

We assume that $X$ supports a $(1,1)$-Poincar\'e inequality,
meaning that there are constants $C_P>0$ and $\lambda \ge 1$ such that for every
ball $B(x,r)$, for every locally integrable function $u$ on $X$,
and for every upper gradient $g$ of $u$, we have 
\[
\vint_{B(x,r)}|u-u_{B(x,r)}|\, d\mu 
\le C_P r\vint_{B(x,\lambda r)}g\,d\mu,
\]
where 
\[
u_{B(x,r)}:=\vint_{B(x,r)}u\,d\mu :=\frac 1{\mu(B(x,r))}\int_{B(x,r)}u\,d\mu.
\]

Given a set $E\subset X$ of finite perimeter, for $\mathcal H$-a.e. $x\in \partial^*E$ we have
\begin{equation}\label{eq:density of E}
\gamma \le \liminf_{r\to 0^+} \frac{\mu(E\cap B(x,r))}{\mu(B(x,r))} \le \limsup_{r\to 0^+} \frac{\mu(E\cap B(x,r))}{\mu(B(x,r))}\le 1-\gamma,
\end{equation}
where $\gamma \in (0,1/2]$ only depends on the doubling constant and the constants in the Poincar\'e inequality, 
see~\cite[Theorem 5.4]{A1}. We denote the set of all such points by $\Sigma_\gamma E$.

For any open set $\Om\subset X$, any $\mu$-measurable set $E\subset X$ with $P(E,\Om)<\infty$, and any Borel set $A\subset \Om$, we know that 
\begin{equation}\label{eq:def of theta}
\Vert D\ch_{E}\Vert(A)=\int_{\partial^{*}E\cap A}\theta_E\,d\mathcal H,
\end{equation}
where
$\theta_E\colon \Om\cap \partial^*E\to [\alpha,C_d]$, with $\alpha=\alpha(C_d,C_P,\lambda)>0$, see~\cite[Theorem 5.3]{A1} 
and~\cite[Theorem 4.6]{AMP}.

The \emph{jump set} of $u\in \BV(X)$ is the set 
\[
S_{u}:=\{x\in X:\, u^{\wedge}(x)<u^{\vee}(x)\},
\]
where $u^{\wedge}(x)$ and $u^{\vee}(x)$ are the lower and upper approximate limits of $u$ defined respectively by
\begin{equation}\label{eq:lower approximate limit}
u^{\wedge}(x):
=\sup\left\{t\in\overline\R:\,\lim_{r\to 0^+}\frac{\mu(B(x,r)\cap\{u<t\})}{\mu(B(x,r))}=0\right\}
\end{equation}
and
\begin{equation}\label{eq:upper approximate limit}
u^{\vee}(x):
=\inf\left\{t\in\overline\R:\,\lim_{r\to 0^+}\frac{\mu(B(x,r)\cap\{u>t\})}{\mu(B(x,r))}=0\right\}.
\end{equation} 

By~\cite[Theorem 5.3]{AMP}, the variation measure of a $\BV$ function can be decomposed into the absolutely 
continuous and singular part, and the latter into the Cantor and jump part, as follows. Given an open set 
$\Omega\subset X$ and $u\in \BV(\Omega)$, we have for any Borel set $A\subset \Om$
\begin{equation}\label{eq:decomposition}
\begin{split}
\Vert Du\Vert(A)
&=\Vert Du\Vert^a(A)+\Vert Du\Vert^s(A)\\
&=\Vert Du\Vert^a(A)+\Vert Du\Vert^c(A)+\Vert Du\Vert^j(A)\\
&=\int_{A}a\,d\mu+\Vert Du\Vert^c(A)
   +\int_{A\cap S_u}\int_{u^{\wedge}(x)}^{u^{\vee}(x)}\theta_{\{u>t\}}(x)\,dt\,d\mathcal H(x),
\end{split}
\end{equation}
where $a\in L^1(\Omega)$ is the density of the absolutely continuous part and the functions $\theta_{\{u>t\}}$ 
are as in~\eqref{eq:def of theta}.

\begin{definition}\label{def:Madayan-traces}
Let $\Omega\subset X$ be a  $\mu$-measurable set
and let $u$ be a $\mu$-measurable
function on $\Omega$.  Let $N_\Om$ be the collection of all points
$x\in \partial\Om$ for which there is some $r>0$ with $\mu(B(x,r)\cap\Om)=0$.
A function $T_{+}u\colon\partial\Omega\setminus N_\Om\rightarrow\R$ is the interior trace of $u$ if for 
$\mathcal H$-a.e. $x\in \partial \Omega$ we have
\[
\lim_{r\rightarrow 0^+}\vint_{\Omega\cap B(x,r)}|u-T_{+}u(x)|\,d\mu=0.
\]
\end{definition}

Note that if $\Om$ is an open set, then $N_\Om$ is empty. Furthermore, we have
$N_{X\setminus\Om}\subset\partial\Om\setminus\partial^*\Om$.

\begin{definition}\label{def:BV0}
Given an open set $U\subset X$, the family $\BV_c(U)$ is the collection of all functions $u\in \BV(X)$ whose
support is a compact subset of $U$. By $\BV_0(U)$ we mean the collection of all functions $u\in \BV(U)$ for which
$T_+u$ exists and $T_+u=0$ $\mathcal{H}$-a.e.~in $\partial U$.
\end{definition}


\begin{definition}\label{def:p+}
Given an open set $\Omega\subset X$ and an open set $U\subset X$, we define
\[
P_+(\Omega,U):=\inf \left\{\liminf_{i\to\infty}\int_{U}g_{\Psi_i}\, d\mu\right\},
\]
where each $g_{\Psi_i}$ is the minimal $1$-weak upper gradient of $\Psi_i$ in $U$, and
where
the infimum is taken over all sequences $(\Psi_i)\subset\liploc(U)$ such that
 $\Psi_i-\ch_{\Omega}\to 0$ in
$L^1(U)$ and $\Psi_i=0$ in 
$U\setminus\Omega$ for each $i\in\N$.

Furthermore, for any $A\subset X$ we let
\[
P_+(\Omega,A):=\inf \left\{P_+(\Omega,U):\, U\textrm{ open},\, A\subset U\right\}.
\]
\end{definition}

In the Appendix we show that if $P_+(\Omega,X)<\infty$, then $P_+(\Omega,\cdot)$ is a 
Radon measure on $X$, which we call the \emph{inner perimeter measure} of $\Omega$.

Note that $P(\Om,A)\le P_+(\Om,A)$ for any $A\subset X$. We will show in
Section~6 that the two quantities $P(\Om,X)$ and $P_+(\Om,X)$ are in fact comparable
when $\Om$ is open and bounded and satisfies the \emph{exterior measure density} condition
\begin{equation}\label{eq:ext-meas-dens-cond}
\limsup_{r\to 0^+}\frac{\mu(B(x,r)\setminus\Omega)}{\mu(B(x,r))}>0\quad\textrm{for }
\mathcal H\textrm{-a.e. }x\in \partial\Omega.
\end{equation}

\begin{definition}
Let $1<p<\infty$ and let $\Om\subset X$ be a nonempty bounded open set
with $\capa_p(X\setminus\Om)>0$.
A function $u\in N^{1,p}(\Om)$ is said to be $p$-harmonic in
$\Om$ if whenever $\phi\in N^{1,p}(X)$ with $\phi=0$ in $X\setminus\Om$, we have
\[
\int_\Om g_u^p\, d\mu\le \int_\Om g_{u+\phi}^p\, d\mu.
\]
Given $f\in N^{1,p}(X)$, we say that a function $u$ is a $p$-harmonic solution to the Dirichlet problem
in $\Om$ with boundary data $f$ if $u\in N^{1,p}(X)$, $u$ is $p$-harmonic in $\Om$, and 
$u=f$ in $X\setminus\Om$.
\end{definition}

The direct method of the calculus of variation yields existence of $p$-harmonic solutions to the Dirichlet problem
($p>1$); see~\cite{S2, BB} for this fact and 
for more on $p$-harmonic functions.
If $f\colon\partial\Om\to\R$ is a Lipschitz function and $\Om$ is bounded, we can extend $f$ to a boundedly supported
Lipschitz function on $X$; such a function is necessarily in $N^{1,p}(X)$ for all $p\ge 1$. Thus we can also talk about 
solutions to the Dirichlet problem with Lipschitz boundary data $f\colon\partial\Om\to\R$. In this paper we will always assume
that the boundary data is a boundedly supported Lipschitz function on $X$.

We will often assume that $\capa_1(X\setminus\Om)>0$, because then
$\capa_p(X\setminus\Om)>0$ for all $p>1$. This follows from the fact that if
$\capa_p(X\setminus\Om)=0$, then $\Vert \ch_{X\setminus\Om}\Vert_{N^{1,p}(X)}=0$
by~\cite[Proposition 1.61]{BB},
and so $\Vert \ch_{X\setminus\Om}\Vert_{N^{1,1}(X)}=0$ by
Remark~\ref{rmk:minimal weak upper gradients}.

\begin{definition}\label{def:innerPerim}
Let $\Om\subset X$ be a an open set. We say that a function $u\in \BV(\Om)$ is a function of least gradient in $\Om$ if
whenever $\phi\in \BV_c(\Om)$, we have
\[
\Vert Du\Vert(\Om)\le \Vert D(u+\phi)\Vert(\Om).
\]
\end{definition}

The principal objects of study in this paper are functions of least gradient as defined above.

\section{Convergence to a function of least gradient}\label{sec:limit is least gradient function}

In this section we show that if there is an $L^1$-convergent sequence $(u_p)$ of $p$-harmonic functions with $p\to 1^+$, 
then the limit is a function of least gradient.
In this section, $g_{u_p}$ always denotes the minimal $p$-weak upper gradient of
$u_p\in N^{1,p}(X)$ on $X$.  If $g_p$ is the minimal $1$-weak upper gradient of $u_p$
on $X$, then
for any open set $U\subset X$, by the fact that locally Lipschitz functions are dense
in $N^{1,1}(U)$ (see~\cite[Theorem 5.47]{BB}) and by Remark
\ref{rmk:minimal weak upper gradients}, we have
\begin{equation}\label{eq:comparing variation and upper gradients}
\Vert Du_p\Vert(U)\le \int_{U}g_p\,d\mu\le \int_{U}g_{u_p}\,d\mu.
\end{equation}
For a Lipschitz function $f$, $g_f$ will denote the minimal $p$-weak upper gradient of
$f$ for any $p>1$. Observe from Remark~\ref{rmk:minimal weak upper gradients} that $g_f$ is indeed independent
of the choice of $p$.

First we note that while we do not know whether a sequence of $p$-harmonic functions is
$L^1$-convergent as $p\to 1^+$, a convergent subsequence always exists.

\begin{lemma}\label{lem:bv embedding lemma}
Let $\Om\subset X$ be a nonempty bounded open set with $\capa_1(X\setminus\Om)>0$,
and let $f\in\Lip(X)$ be boundedly supported.
For each $p>1$, let $u_p\in N^{1,p}(X)$ be 
a \p-harmonic function in $\Om$ such that $u_p|_{X\setminus\Om}=f$. 
Then there exists a sequence $p_k\to 1^+$
such that
$u_{p_k}\rightarrow u$ in $L^1(X)$ as $k\to\infty$
for some $u\in \BV(X)$.
\end{lemma}

\begin{proof}
By the maximum principle for the Dirichlet problem for $p$-harmonic functions, 
$\|u_{p}\|_{L^\infty(X)}\leq\|f\|_{L^{\infty}(X)}$, and so for all $p>1$
\begin{align*}
\Vert u_p\Vert_{L^1(X)}
&\le \|u_{p}\|_{L^\infty(X)}\mu(\Om)+
\Vert f\Vert_{L^1(X\setminus\Om)}\\
&\le \|f\|_{L^\infty(X)}\mu(\Om)+\Vert f\Vert_{L^1(X)}<\infty.
\end{align*}
Let $L$ be the global Lipschitz constant of $f$. Then
\begin{align*}
 \int_\Omega g_{u_p}\, d\mu\le \left(\int_\Omega g_{u_p}^p\, d\mu\right)^{1/p}\mu(\Omega)^{1-1/p} 
& \le \left(\int_\Omega g_{f}^p\, d\mu\right)^{1/p}\mu(\Omega)^{1-1/p}\\
  &  \le L \mu(\Omega)^{1-1/p}.
\end{align*}
On the other hand,
\[
\int_{X\setminus\Omega}g_{u_p}\,d\mu=\int_{X\setminus\Omega}g_{f}\,d\mu,
\]
see~\cite[Lemma 2.19]{BB}.
Thus by~\eqref{eq:comparing variation and upper gradients},
\[
\Vert Du_p\Vert(X)\le L \mu(\Omega)^{1-1/p}+\int_{X\setminus\Omega}g_{f}\,d\mu.
\]
We conclude that the sequence $(u_p)_p$ is a bounded sequence in $\BV(X)$, 
and so by the compact embedding given in~\cite[Theorem~3.7]{Miranda03},
a subsequence converges in $L^1_{\loc}(X)$ and hence
in $L^1(X)$ to some function $u\in L^1(X)$, and by the lower semicontinuity
of the total variation, we have $u\in\BV(X)$.
\end{proof}

\begin{theorem}\label{thm:main1}
Let $\Om\subset X$ be a nonempty bounded open set with $\capa_1(X\setminus\Om)>0$,
and let $f\in\Lip(X)$ be boundedly supported. For each $p>1$ let $u_p\in N^{1,p}(X)$ be a 
\p-harmonic function in $\Om$ such that $u_p|_{X\setminus\Om}=f$. Suppose that $(u_p)_{p>1}$ is a
sequence of such \p-harmonic functions and that $u_p \to u$ in $L^1(X)$ as $p\to 1^+$. Then $u$ is a 
function of least gradient in $\Om$.
\end{theorem}

\begin{proof}
By the proof of Lemma~\ref{lem:bv embedding lemma}, we have $u\in \textrm{BV}(X)$.
Let $\psi\in \textrm{BV}_c(\Om)$ and $K:=\spt(\psi)$. Clearly
\[
\int_{\Omega}g_{u_{p}}^p\,d\mu\leq \int_{\Omega}g_{f}^p\,d\mu\leq L^p\mu(\Omega),
\]
where  $L$ is the global Lipschitz constant of $f$, and therefore  $(g_{u_p}^p)_{1<p<2}$ is uniformly bounded in $L^1(\Om)$. Consequently, there exists a subsequence, still written as $(g_{u_p}^p)_{p>1}$, and a positive Radon measure of finite mass $\nu$ on $\Om$ such that
\[
g_{u_p}^p\, d\mu \to d\nu \quad \textrm{weakly* in } \Om \textrm{ as } p\to 1^+.
\]
We now choose $\widetilde K\Subset \Omega$ such that $K\subset \widetilde K$ and $\nu(\partial \widetilde K)=0$. For small enough $\eps>0$, 
\[
\widetilde K^\eps := \bigcup_{x\in \widetilde K}B(x,\eps) \Subset \Om.
\]
We fix $\eta\in\Lip(X)$ such that $0\leq\eta\leq 1$, 
\[
\eta =1\text{ in } \widetilde K, \quad  \eta = 0\text{ in }X\setminus \widetilde K^{\eps/2}, \quad\text{and }\ g_\eta\le 2/\eps.
\]
As $u+\psi\in \BV(\widetilde K^\eps)$, there exists a sequence
$(\Psi_k)\subset \liploc(\widetilde K^\eps)$ such that $\Psi_k \to u+\psi$ in
$L^1(\widetilde K^\eps)$ and
\begin{equation}\label{eq:variation of u plus psi}
\|D(u+\psi)\|(\widetilde K^\eps) = \lim_{k\to\infty}\int_{\widetilde K^\eps}g_{\Psi_k}\, d\mu,
\end{equation}
where $g_{\Psi_k}$ is the minimal $p$-weak upper gradient of $\Psi_k$ in
$\widetilde K^\eps$, for $p>1$, see the discussion on page \pageref{p-weak upper gradients in total variation}.
We set
\[
\psi_{k,p} := \eta\Psi_k + (1-\eta)u_p.
\]
Then $\psi_{k,p}=u_p$ in $X\setminus \widetilde K^{\eps/2}$ and $\psi_{k,p}=\Psi_k$ in $\widetilde K$. By the Leibniz rule given in~\cite[Lemma~2.18]{BB},
\begin{align*}
g_{\psi_{k,p}} &\le 
g_{\Psi_k}\eta + g_{u_p}(1-\eta) + 
g_\eta|\Psi_k-u_p|\\
&\le g_{\Psi_k}\ch_{\widetilde K^{\eps/2}} + g_{u_p}\ch_{X\setminus \widetilde K} +
 (2/\eps)|\Psi_k-u_p|\ch_{\widetilde K^{\eps/2}\setminus \widetilde{K}}.
\end{align*}
Since $u_p$ is \p-harmonic, we have
\begin{align*}
\left(\int_{\widetilde{K}^\eps}g_{u_p}^p\, d\mu\right)^{1/p} & \leq \left(\int_{\widetilde{K}^\eps}g_{\psi_{k,p}}^p\, d\mu\right)^{1/p} \\
& \leq \left(\int_{\widetilde{K}^{\eps/2}}g_{\Psi_k}^p\, d\mu\right)^{1/p} + \left(\int_{\widetilde{K}^\eps\setminus \widetilde{K}}g_{u_p}^p\, d\mu\right)^{1/p} \\
& \qquad \quad +\frac{2}{\eps} \left(\int_{\widetilde{K}^{\eps/2}\setminus \widetilde{K}}\vert \Psi_k-u_p\vert^p\, d\mu\right)^{1/p}.
\end{align*}
Therefore, by lower semicontinuity and~\eqref{eq:comparing variation and upper gradients} we get
\begin{align*}
\|Du\|(\widetilde{K}^\eps)
&\le \liminf_{p\to 1^+}\|Du_p\|(\widetilde{K}^\eps)\\
& \leq \liminf_{p\to 1^+}\int_{\widetilde{K}^\eps}g_{u_p}\, d\mu \\
& \leq \liminf_{p\to 1^+}\mu(\widetilde{K}^\eps)^{1-1/p}\left(\int_{\widetilde{K}^\eps}g_{u_p}^p\, d\mu\right)^{1/p} \\ 
& \leq \int_{\widetilde{K}^{\eps/2}}g_{\Psi_k}\, d\mu+ \limsup_{p\to 1^+}\left(\int_{\widetilde{K}^\eps\setminus \widetilde{K}}g_{u_p}^p\, d\mu\right)^{1/p} \\
& \qquad \quad + \frac{2}{\eps}\int_{\widetilde{K}^{\eps/2}\setminus \widetilde{K}}\vert \Psi_k-u\vert\, d\mu,
\end{align*}
which in turn leads to
\[
\|Du\|(\widetilde{K}) \leq \int_{\widetilde{K}^{\eps/2}}g_{\Psi_k}\, d\mu
 + \nu(\overline{\widetilde{K}^\eps}\setminus \widetilde{K}) 
  + \frac{2}{\eps} \int_{\widetilde{K}^{\eps/2}\setminus \widetilde{K}}\vert \Psi_k-u\vert\, d\mu.
\]
Letting $k\to \infty$, we get by~\eqref{eq:variation of u plus psi}
\[
\|Du\|(\widetilde{K}) \le \|D(u+\psi)\|(\widetilde K^\eps)+
\nu(\overline{\widetilde{K}^\eps}\setminus \widetilde{K}).
\]
Then letting $\eps\to 0$, by the fact that $\nu(\partial \widetilde{K})=0$ we get 
\[
\|Du\|(\widetilde{K}) \leq \|D(u+\psi)\|(\widetilde{K}).
\]
The claim follows from this.
\end{proof}

%

\section{Definitions of the Dirichlet problem for $p=1$}

The focus of this paper is to show that the limit of $p$-harmonic functions with 
Lipschitz boundary data $f$, as $p\to 1^+$, solves a reasonable notion of a Dirichlet problem with boundary 
data $f$. The issue is to give such a notion.
In the case of 
the $p$-energy, there is no ambiguity in the sense in which we want to fix the boundary values
of the function, if the boundary values are themselves restrictions
of Newton-Sobolev functions.
In the case $p=1$, we propose
the following two ways of defining solutions to the Dirichlet problem.

\begin{definition}\label{def:dirichlet}
Let $\Omega\subset X$ be a nonempty bounded open set
 with $\capa_1(X\setminus\Om)>0$, and let $f\in \Lip(X)$
be boundedly supported. We say that a function $u$
is a solution to the Dirichlet problem for functions of least gradient with boundary data $f$ in the sense of~(B) (respectively
in the sense of~(T)) if it is a solution to the following minimization problem:
\begin{itemize}
\item[(B)] Minimize $\Vert Dv\Vert(\overline{\Omega})$ over all functions $v\in \BV(X)$ with $v=f$ on 
$X\setminus\overline{\Omega}$,
\item[(T)] Minimize $\Vert Dv\Vert(\Omega)+\int_{\partial\Omega}|T_+v-f|(x)\, dP_+(\Omega,x)$ over all 
functions $v\in \BV(\Omega)$.
\end{itemize}
\end{definition}

Note that in definition (T), we need to make extra assumptions on $\Om$
to ensure that the boundary integral is well defined.
In both definitions, the solution is allowed to have jumps on the boundary of $\Omega$.
In definition (B),  this is
taken into account by including the variation measure from the boundary $\partial \Omega$ as well. 
The advantage of this approach is that its energy is 
more straightforward to calculate, and we need fewer assumptions on $\Om$.
The 
drawback is that  contrary to the formulation (T), the structure of the underlying space $X$ close to the  boundary but 
outside $\Omega$ also affects the minimization problem. For instance, let $X$ be the Euclidean
space $\R^n$ equipped with the Euclidean metric, and let $\Om$ be the unit ball
centered at the origin. Let
$\alpha\in(0,1]$ and equip $X$ with the measure
\[
d\mu_\alpha:=(\ch_\Om+\alpha\ch_{\R^n\setminus\Om})\,d\mathcal{L}^n,
\]
where $\mathcal L^n$ is the $n$-dimensional Lebesgue measure.
It can be shown that for $u\in\BV(X)$, $\Vert Du\Vert(\overline{\Om})=\Vert D_{\text{Euc}} u\Vert(\Om)+\alpha\Vert D_{\text{Euc}} u\Vert(\partial\Om)$, where
$\Vert D_{\text{Euc}} u\Vert$ is the total variation with respect to $\mathcal L^n$.
Similarly, in this setting we have $P_+(\Om,X)=2\pi$ but $P(\Om,X)=2\alpha\pi$.

\section{Discrete convolutions}

A tool that is commonly used in analysis on metric spaces is the \emph{discrete convolution}.
Given any  open set $U\subset X$ and a scale $R>0$, we can choose a \emph{Whitney-type covering} 
$\{B_j=B(x_j,r_j)\}_{j=1}^{\infty}$ of $U$ such that (see e.g.~\cite[Theorem 3.1]{BBS})
\begin{enumerate}
\item for each $j\in\N$,
\[
r_j= \min\left\{\frac{\dist(x_j,X\setminus U)}{40\lambda},\,R\right\},
\]
\item for each $k\in\N$, the ball $10\lambda B_k$ intersects at most $C_0=C_0(C_d,\lambda)$ 
balls $10\lambda B_j$ (that is, a bounded overlap property holds),
\item  
if $10\lambda B_j$ intersects $10\lambda B_k$, then $r_j\le 2r_k$.
\end{enumerate}

Given such a covering of $U$, 
we can take a partition of unity $\{\phi_j\}_{j=1}^{\infty}$ subordinate to the covering, such that $0\le \phi_j\le 1$, each 
$\phi_j$ is a $C/r_j$-Lipschitz function, and $\supp(\phi_j)\subset 2B_j$ for each 
$j\in\N$ (see e.g.~\cite[Theorem 3.4]{BBS}). Finally, we can define the \emph{discrete convolution} $v$ of 
any $u\in L_{\loc}^1(U)$ with respect to the Whitney-type covering by
\[
v:=\sum_{j=1}^{\infty}u_{5B_j}\phi_j.
\]
In general, $v\in\liploc(U)$, and hence $v\in L^1_{\loc}(U)$.

Let $v$ be the discrete convolution of $u\in L^1_{\loc}(U)$ with
$\Vert Du\Vert(U)<\infty$,
with respect to a Whitney-type covering $\{B_j\}_{j\in\N}$ of $U$ at scale $R$. Then $v$ has a local Lipschitz constant
\begin{equation}\label{eq:upper gradient of discrete convolution}
\Lip v\le C_{\textrm{lip}}\sum_{j=1}^{\infty}\ch_{B_j}\frac{\Vert Du\Vert(10\lambda B_j)}{\mu(B_j)},
\end{equation}
with $C_{\textrm{lip}}$ depending only on the doubling constant of the measure and the constants in the Poincar\'e inequality, see e.g. the proof of~\cite[Proposition~4.1]{KKST3}. From this it follows by the bounded overlap property (2) that
\begin{equation}\label{eq:total variation of discrete convolution}
\int_{U}\Lip v \,d\mu\le C_0 C_{\textrm{lip}}\Vert Du\Vert(U).
\end{equation}
Moreover (noting that $v$ depends on the scale $R$),
\begin{equation}\label{eq:L1 approximation property of DC}
\Vert v-u\Vert_{L^1(U)}\to 0\quad\textrm{as}\quad R\to 0,
\end{equation}
see the proof of~\cite[Proposition 4.1]{KKST3};
note that $u$ does not need to be in $L^1(U)$, only in $L_{\loc}^1(U)$.

Now let $(v_i)$ be a sequence of discrete convolutions of $u\in \BV_{\loc}(U)$
with respect to Whitney-type coverings at scales $R_i\searrow 0$.
According to~\cite[Proposition 4.1]{KKST3}, we have for some constant
$\widetilde{\gamma}\in (0,1/2]$
\begin{equation}\label{eq:pointwise convergence of discrete convolutions}
\begin{split}
&(1-\widetilde{\gamma}) u^{\wedge}(y)+ \widetilde{\gamma} u^{\vee}(y)
\le \liminf_{i\to \infty}v_i(y)\\
&\qquad\qquad\qquad\qquad \le \limsup_{i\to \infty}v_i(y)
\le \widetilde{\gamma} u^{\wedge}(y)+ (1-\widetilde{\gamma})u^{\vee}(y)
\end{split}
\end{equation}
for $\mathcal H$-a.e. $y\in U$; recall the definitions of the lower and upper approximate
limits from~\eqref{eq:lower approximate limit} and~\eqref{eq:upper approximate limit}.

By applying discrete convolutions, we can show that in the definition of the
total variation, we can replace convergence in $L^1_{\loc}(\Omega)$ with
convergence in $L^1(\Omega)$.

\begin{lemma}\label{lem:L1 loc and L1 convergence}
Let $\Omega\subset X$ be an open set and let $u\in L^1_{\loc}(\Omega)$
with $\Vert Du\Vert(\Omega)<\infty$. Then there exists a sequence of functions
$(w_i)\subset \liploc(\Omega)$ with $w_i-u\to 0$ in $L^1(\Omega)$ and 
$\int_\Omega g_{w_i}\, d\mu\to \Vert Du\Vert(\Omega)$, where each $g_{w_i}$ is the minimal
$1$-weak upper gradient of $w_i$.
\end{lemma} 

Note that we cannot write $w_i\to u$ in $L^1(\Om)$, if we do not have $u\in L^1(\Om)$.

\begin{proof} 
For every $\delta>0$, let
\[
\Omega_{\delta}:=\{y\in\Omega:\,\dist(y,X\setminus\Omega)>\delta\}. 
\]
Fix $\eps>0$ and $x\in X$, and choose $\delta\in (0,1)$ such that
\[
\Vert Du\Vert(\Omega\setminus(\Omega_{\delta}\cap B(x,1/\delta)))<\eps.
\]
Let
\[
\eta(y):=\max\left\{0,1-\frac 4\delta \dist(y,\Omega_{\delta/2}\cap B(x,2/\delta))\right\},
\]
which is a $4/\delta$-Lipschitz function.

Let each $v_i\in\liploc(\Omega)$ be a discrete convolution of $u$ in $\Omega$, at scale 
$1/i$. From the definition of the total variation we get a sequence of 
functions $u_i\in\liploc(\Omega)$ with $u_i\to u$ in $L_{\loc}^1(\Omega)$ and
\[
\int_\Omega g_{u_i}\, d\mu \to\Vert Du\Vert(\Omega).
\]
Now define
\[
w_i:=\eta u_i+(1-\eta)v_i,
\]
so that $w_i- u\to 0$ in $L^1(\Omega)$  by~\eqref{eq:L1 approximation property of DC}, and by the Leibniz rule of~\cite[Lemma~2.18]{BB}, 
\[
g_{w_i}\le g_{u_i}\eta +g_{v_i}(1-\eta)+g_{\eta}|u_i-v_i|.
\]
Here $g_{w_i}$, $g_{u_i}$, $g_{v_i}$, and $g_{\eta}$ all denote minimal $1$-weak upper gradients.
Since $g_{\eta}=0$ outside $\Omega_{\delta/4}\cap B(x,4/\delta)\Subset \Omega$, we have 
$g_{\eta}|u_i-v_i|\to 0$ in $L^1(\Omega)$, and by also  
using~\eqref{eq:upper gradient of discrete convolution}, we get
\begin{align*}
\limsup_{i\to\infty}\int_{\Omega}g_{w_i}\,d\mu
&\le \limsup_{i\to\infty}\int_{\Omega}g_{u_i}\,d\mu
    +\limsup_{i\to\infty}\int_{\Omega\setminus (\Omega_{\delta/2}\cap B(x,2/\delta))}g_{v_i}\,d\mu\\
&\le \Vert Du\Vert(\Omega)+C\Vert Du\Vert(\Omega\setminus(\Omega_{\delta}\cap B(x,1/\delta)))\\
&\le \Vert Du\Vert(\Omega)+C\eps.
\end{align*} 
By a diagonalization argument, where we also let $\varepsilon\to 0$ (and hence $\delta\to 0$), we complete the proof.
\end{proof}

\section{Comparability of $P_+$ and $P$}

Recall the definition of $P_+(\Om,\cdot)$ from Definition~\ref{def:p+}.
As shown by the example found
in the discussion following  Definition~\ref{def:dirichlet}, $P_+(\Om,\cdot)$
does not necessarily agree with $P(\Om,\cdot)$. In light of this, the current section aims to compare
$P_+(\Om,\cdot)$ and $P(\Om,\cdot)$.
The main
result of this section is Theorem~\ref{thm:P+vsP}.

An analog of $P_+(\Om,X)$ was studied in~\cite{Sch}, where it was shown that for certain open sets
$\Om\subset\R^n$, one has $P_+(\Om,\R^n)=P(\Om,\R^n)$.
More precisely, it was shown that in the Euclidean setting, if an open set
$\Om\subset\R^n$
satisfies $\mathcal{H}^{n-1}(\partial \Om\setminus\partial^*\Om)=0$, then it is possible to
find open sets $\Om_i\Subset \Om$ with $\Om=\bigcup_{i\in\N}\Om_i$ and
$\mathcal H^{n-1}(\partial\Om_i)\to P(\Om,\R^n)$, where $\mathcal H^{n-1}$ is the $(n-1)$-dimensional
Hausdorff measure.  We obtain in Corollary~\ref{cor:approximation finite perimeter from inside}
a weak analog of this result. In fact, our corollary is applicable to a wider class of
Euclidean domains than the result of~\cite{Sch},
since we can permit the
part of the boundary in which $\Om$ is ``thin" to be very large.

In the following lemma, we essentially follow an argument that can be found e.g. in
\cite[p. 67]{Rud}.

\begin{lemma}\label{lem:using Mazurs lemma}
Let $K\subset X$ be compact, and let $\alpha\in [0,1)$ and $\eps>0$. Take a sequence $(v_i)\subset C(K)$ with
$0\le v_i\le 1$ for every $i\in\N$, and
\[
\limsup_{i\to\infty}v_i(x)\le \alpha
\]
for every $x\in K$. Then there exists a convex combination of $v_i$, denoted
by $\widehat{v}$, such that 
 $\widehat{v}(x)\le \alpha+\eps$ for every $x\in K$.
\end{lemma}

\begin{proof}
We have
\[
\lim_{i\to\infty}\max\{v_i(x),\alpha\}=\alpha
\]
for every $x\in K$.
Note that the functions $\max\{v_i(x),\alpha\}$, and the constant function $\alpha$, are continuous and take values between $0$ and $1$. Thus 
for any signed Radon measure $\nu$ on $K$ we have by Lebesgue's dominated convergence theorem that
\[
\int_K \max\{v_i(x),\alpha\}\,d\nu\to \int_K \alpha\,d\nu.
\]
Since $K$ is compact, we have $C(K)=C_c(K)$ and then by the Riesz representation theorem
we conclude that
$\max\{v_i(x),\alpha\}\to \alpha$ weakly in the space $C(K)$. By Mazur's lemma, see~\cite[Theorem 3.13]{Rud}, we can find convex combinations of
the functions $w_i:=\max\{v_i,\alpha\}$, denoted by $\widehat{w}_i$, which converge strongly in the space $C(K)$ to $\alpha$. In other words, $\widehat{w}_i\to \alpha$ \emph{uniformly} in $K$.
Thus for a sufficiently large choice of
$i\in\N$, we have
$\widehat{w}_i(x)\le \alpha+\eps$ for all $x\in K$. 
With $\widehat{w}_i=\sum_{j=1}^N\lambda_{i,j} w_j$ for some $N\in\N$ and the
appropriate choice of
numbers $\lambda_{i,j}\in[0,1]$ such that $\sum_{j=1}^N\lambda_{i,j}=1$, we set
$\widehat{v}=\sum_{j=1}^N\lambda_{i,j} v_j$.
\end{proof}

\begin{proposition}\label{prop:approximation with Lip functions from inside}
Let $\Omega,U\subset X$ be open sets with $P(\Omega,U)<\infty$, and suppose that there exists
$A\subset \partial\Omega\cap U$ with $\mathcal H(A)<\infty$ such that
\begin{equation}\label{eq:exterior measure density}
\limsup_{r\to 0^+}\frac{\mu(B(x,r)\setminus\Omega)}{\mu(B(x,r))}>0
\end{equation}
for every $x\in \partial\Omega\cap U\setminus A$.
Let $U'\Subset U$.
Then there exists a sequence $(w_k)\subset \liploc(U)$ such that  for each $k\in\N$, $w_k=0$ in
$U'\setminus \Omega$, $w_k\to \ch_{\Omega}$ in $L^1(U)$ and
\[
\limsup_{k\to\infty}\int_{U} \Lip w_k\,d\mu\le C_{in} (P(\Omega,U)+\mathcal H(A))
\]
for a constant $C_{in}=C_{in}(C_d,C_P,\lambda)$.
\end{proposition}


\begin{remark}\label{rem:bad-example}
If $X=\R^2$ (unweighted) and $\Om$ is the slit disk
\[
\Om=\{(x,y)\in\R^2:\, x^2+y^2<1\}\setminus[-1,0]\times\{0\},
\]
and $U$ is the unit disk, we have $P(\Omega,U)=0$
and the set $A$ can be taken to be the slit. If we add countably many slits,
see Example~\ref{ex:countably many slits} below, we still have $P(\Omega,U)=0$ but
$\mathcal H(A)=\infty$ and the conclusion of the proposition becomes meaningless.
\end{remark}

\begin{proof}[Proof of Proposition~\ref{prop:approximation with Lip functions from inside}]
Let each $v_i$, $i\in\N$, be the discrete convolution of $\ch_{\Omega}$ with respect to a Whitney-type covering of $U$ at scale $1/i$. We can add to the set $A$ the $\mathcal H$-negligible set where~\eqref{eq:pointwise convergence of discrete convolutions} fails with $u=\ch_{\Omega}$.
Fix $\eps>0$.
We can pick balls $B(x_j,s_j)$ intersecting $A$ with $s_j\le \eps$,
\[
A\subset \bigcup_{j\in \N}B(x_j,s_j),
\]
and
\[
\sum_{j\in\N}\frac{\mu(B(x_j,s_j))}{s_j}\le \mathcal H(A)+\eps.
\]
Furthermore, we can choose radii $r_j\in [s_j,2s_j]$ such that
\[
P(B(x_j,r_j),X)\le C\frac{\mu(B(x_j,r_j))}{r_j}
\]
for each $j\in\N$, see~\cite[Lemma 6.2]{KKST1}.

For brevity, let us write $B_j:=B(x_j,r_j)$, $j\in\N$. Then by the subadditivity~\eqref{eq:Caccioppoli sets form an algebra} and the lower semicontinuity of perimeter, we have
\begin{equation}\label{eq:perimeter of modified Omega}
\begin{split}
P\left(\Omega\setminus \bigcup_{j\in\N}B_j,U\right)
&\le P(\Omega,U)+\sum_{j\in\N}P(B_j,X)\\
&\le P(\Omega,U)+C\sum_{j\in\N}\frac{\mu(B_j)}{r_j}\\
&\le P(\Omega,U)+C\mathcal H(A)+C\eps.
\end{split}
\end{equation}
Let each $\breve{v}_i$, $i\in\N$, be the discrete convolution of $\ch_{\Omega\setminus \bigcup_{j\in\N}B_j}$ with respect to the same Whitney-type covering of $U$ at scale $1/i$  used also in defining the functions $v_i$.
By the properties~\eqref{eq:total variation of discrete convolution}
and~\eqref{eq:L1 approximation property of DC} of discrete convolutions, we have
$\breve{v}_i- \ch_{\Omega\setminus \bigcup_{j\in\N}B_j}\to 0$ in $L^1(U)$ and
\begin{equation}\label{eq:first local Lipschitz constant estimate}
\Vert \Lip \breve{v}_i \Vert_{L^1(U)}\le C_0 C_{\textrm{lip}} P\left(\Omega\setminus \bigcup_{j\in\N}B_j,U\right)
\end{equation}
for each $i\in\N$.
Note that $\breve{v}_i(x)\le v_i(x)$ for every $x\in U$.
Thus for every $x\in \partial\Omega\cap U\setminus A$ we have by~\eqref{eq:pointwise convergence of discrete convolutions}
\begin{align*}
\limsup_{i\to\infty}\breve{v}_i(x)
\le \limsup_{i\to\infty} v_i(x)
\le \widetilde{\gamma} \ch_{\Omega}^{\wedge}(x)+ (1-\widetilde{\gamma})\ch_{\Omega}^{\vee}(x)
\le 1-\widetilde{\gamma};
\end{align*}
note that $\ch_{\Omega}^{\wedge}(x)=0$ for every $x\in \partial\Omega\cap U\setminus A$ by~\eqref{eq:exterior measure density}.
Moreover, $\lim_{i\to\infty}\breve{v}_i(x)=0$ for every $x\in A$,
since $\ch_{\Omega\setminus \bigcup_{j\in\N}B_j}=0$ in a neighborhood of every $x\in A$.
Note that $\overline{U'}\setminus \Omega$ is a compact set.
Using Lemma~\ref{lem:using Mazurs lemma}, we find for every $i\in\N$ a convex combination
of the functions $\{\breve{v}_k\}_{k=i}^{\infty}$, denoted by $\widehat{v}_i$, such that
\begin{equation}\label{eq:bound for convex combinations}
\widehat{v}_i(x)\le 1-\widetilde{\gamma}/2 \qquad\textrm{for every }x\in \overline{U'}\setminus \Omega.
\end{equation}
Clearly we still have $\widehat{v}_i\in \liploc(U)$ with
$\widehat{v}_i- \ch_{\Omega\setminus \bigcup_{j\in\N}B_j}\to 0$ in $L^1(U)$, and by~\eqref{eq:first local Lipschitz constant estimate} and the subadditivity~\eqref{eq:subadditivity of local Lipschitz constants},
\begin{equation}\label{eqn:Lipv}
\Vert \Lip \widehat{v}_i \Vert_{L^1(U)}\le C_0 C_{\textrm{lip}} P\left(\Omega\setminus \bigcup_{j\in\N}B_j,U\right).
\end{equation}
Next, let
\[
\breve{w}_i:=\frac{\max\{0,\widehat{v}_i-1+\widetilde{\gamma}/2\}}{\widetilde{\gamma}/2},\qquad i\in\N.
\]
Then by~\eqref{eq:bound for convex combinations}, $\breve{w}_i=0$ in $\overline{U'}\setminus \Omega$.
Again, we still have $\breve{w}_i\in \liploc(U)$ with $\breve{w}_i- \ch_{\Omega\setminus \bigcup_{j\in\N}B_j}\to 0$ in $L^1(U)$, and by~\eqref{eq:perimeter of modified Omega} and~\eqref{eqn:Lipv},
\begin{align*}
\Vert \Lip \breve{w}_i \Vert_{L^1(U)}
&\le 2\widetilde{\gamma}^{-1} C_0 C_{\textrm{lip}} P\left(\Omega\setminus \bigcup_{j\in\N}B_j,U\right)\\
&\le C (P(\Omega,U)+\mathcal H(A)+\eps).
\end{align*}
We can do the above for each $\eps=1/k$, $k\in\N$.
Denote $\Omega_k:=\Omega\setminus \bigcup_{j\in\N}B_j$, with the balls $B_j$ picked corresponding to the choice $\eps=1/k$.
Thus we obtain sequences $\breve{w}_{k,i}$ with $\breve{w}_{k,i}-\ch_{\Omega_k}\to 0$ in
$L^1(U)$ as $i\to\infty$.
Then for each $k\in\N$ we can pick a sufficiently large $i_k\ge k$ such that
\[
\Vert \breve{w}_{k,i_k}-\ch_{\Omega_k}\Vert_{L^1(U)}\le 1/k,
\]
\[
\Vert \Lip \breve{w}_{k,i_k} \Vert_{L^1(U)}\le C ( P(\Omega,U)+\mathcal H(A)+1/k),
\]
and $\breve{w}_{k,i_k}=0$ in $U'\setminus \Omega$.
Since furthermore $\ch_{\Omega_k}- \ch_{\Omega}\to 0$ in $L^1(U)$ as $k\to\infty$, we have
$ \breve{w}_{k,i_k}- \ch_{\Omega}\to 0$ in $L^1(U)$. Finally, we can define $w_k:=\breve{w}_{k,i_k}$, $k\in\N$.
\end{proof}

Note that we always have $P(\Omega,U)\le P_+(\Omega,U)$, since the definition of the latter involves 
a more restricted class of approximating functions. Now  we can show the following.

\begin{theorem}\label{thm:P+vsP}
Let $\Omega\subset X$ be a bounded open set with $P(\Omega,X)<\infty$, and suppose that
\[
\limsup_{r\to 0^+}\frac{\mu(B(x,r)\setminus\Omega)}{\mu(B(x,r))}>0
\]
for $\mathcal H$-a.e. $x\in \partial\Omega$.
Then $P_+(\Om,X)<\infty$, and
for any open set $U\subset X$, we have
$P(\Om, U)\le P_+(\Omega,U)\le C P(\Omega,U)$.
\end{theorem}

\begin{proof}
Take a bounded open set $U'\subset X$ with $\Omega\Subset U'$.
Let $w_k\in\liploc(X)$ be the sequence 
given by Proposition~\ref{prop:approximation with Lip functions from inside} 
with the choice $U=X$ (note that now $\mathcal H(A)=0$).
Then $w_k=0$ in $U'\setminus \Omega$, and in fact from the proof of 
Proposition~\ref{prop:approximation with Lip functions from inside} it is 
easy to see that $w_k=0$ in $X\setminus \Omega$.
Then by the definition of $P_+(\Om,\cdot)$ and the fact that the local Lipschitz constant is an upper gradient,
\begin{align*}
P_+(\Omega,X)
\le \liminf_{k\to\infty} \int_{X}g_{w_k}\,d\mu
\le \liminf_{k\to\infty} \int_{X}\Lip w_k\,d\mu
\le C P(\Omega,X)<\infty.
\end{align*}
Therefore $P_+(\Omega,\cdot)$ is a Radon measure on $X$, see Appendix.
For an open set $U\subset X$, the first inequality of the second claim is clear.
To prove the second inequality, fix $\eps>0$. For some 
$U'\Subset U$ we have $P_+(\Omega,U)\le P_+(\Omega,U')+\eps$.
Let $(w_k)\subset\liploc(U)$ be the sequence 
given by Proposition~\ref{prop:approximation with Lip functions from inside}.
Then
\begin{align*}
P_+(\Omega,U)\le P_+(\Omega,U')+\eps
&\le \liminf_{k\to\infty} \int_{U'}g_{w_k}\,d\mu+\eps\\
&\le \liminf_{k\to\infty} \int_{U'}\Lip w_k\,d\mu+\eps\\
&\le C P(\Omega,U)+\eps.
\end{align*}
By letting $\eps\to 0$, we obtain the result.
\end{proof}

To conclude this section, we prove two corollaries of 
Proposition~\ref{prop:approximation with Lip functions from inside} that will not be needed in the 
sequel, but may be of independent interest. First we need a lemma.

\begin{lemma}\label{lem:coarea inequality}
For any $w\in\Lip_c(X)$,
\[
\int_{-\infty}^{\infty}\mathcal H(\partial\{w>t\})\,dt\le C_{co}\int_{X}\Lip w\,d\mu,
\]
where $C_{co}$ only depends on the doubling constant of the measure.
\end{lemma}
\begin{proof}
By~\cite[Proposition 3.5]{KoLa} (which is based on~\cite{BH}) the following coarea inequality holds:
for any $w\in\Lip_c(X)$,
\[
\int_{-\infty}^{\infty}\mu^+(\partial\{w>t\})\,dt\le \int_{X}\Lip w\,d\mu.
\]
Since $\mathcal H(A)\le C_d^3\mu^+(A)$ for any $A\subset X$ (see e.g.~\cite[Proposition 3.12]{KoLa}), we obtain the result.
\end{proof}


\begin{corollary}\label{cor:approximation finite perimeter from inside}
Let $\Omega\subset X$ be a bounded open set with $P(\Omega,X)<\infty$, and suppose that there exists
$A\subset \partial\Omega$ with $\mathcal H(A)<\infty$ such that
\[
\limsup_{r\to 0^+}\frac{\mu(B(x,r)\setminus\Omega)}{\mu(B(x,r))}>0
\]
for every $x\in \partial\Omega\setminus A$. Then there exists a sequence of open sets
$\Omega_j\Subset \Omega$ with $\ch_{\Om_j}(x)\to 1$ for every $x\in\Om$ and
\[
\mathcal H(\partial \Omega_j)\le C(P(\Omega,X)+\mathcal H(A))
\]
for each $j\in\N$.
\end{corollary}

\begin{proof}
Choose an open set $U'$ with $\Omega\Subset U'\Subset X$. Apply Proposition~\ref{prop:approximation with Lip functions from inside} with $U=X$
to obtain a sequence $w_k\in \Lip_{\loc}(X)$ with $w_k\to \ch_{\Omega}$ in $L^1(X)$,
\[
\int_X \Lip w_k \,d\mu \le C_{in}(P(\Omega,X)+\mathcal H(A)),
\]
and $w_k=0$ in $U'\setminus \Omega$.
From the proof of Proposition~\ref{prop:approximation with Lip functions from inside} it is easy to see that in fact $w_k=0$ in $X\setminus \Omega$, so that $w_k\in \Lip_{c}(X)$ for each $k\in\N$.
From the proof of Proposition~\ref{prop:approximation with Lip functions from inside}
we can also see that $w_k(x)\to 1$ for every $x\in\Omega$, so that
for any $t\in (0,1)$,
$\ch_{\{w_k>t\}}(x)\to 1$
for every $x\in\Omega$.
By Lemma~\ref{lem:coarea inequality},
\[
\int_0^1\mathcal H(\partial\{w_k>t\})\,dt\le C_{co}\int_X \Lip w_k\,d\mu\le C_{co} C_{in}(P(\Omega,X)+\mathcal H(A))
\]
for all $k\in\N$.
Thus for any fixed $k\in\N$ we find a set $T_k\subset (0,1)$ with $\mathcal L^1(T_k)\ge 1/2$ such that
for all $t\in T_k$,
\[
\mathcal H(\partial\{w_k>t\})\le 2C_{co}C_{in}(P(\Omega,X)+\mathcal H(A)).
\]
Now, if for every $t\in (0,1)$ there were an index $N_t\in\N$ such that $t\notin T_k$ for all $k\ge N_t$, then by the Lebesgue dominated convergence theorem we would have
\[
\int_0^1 \ch_{T_k}\,d\mathcal L^1 \to 0,
\]
which is a contradiction. Thus there exists $t\in (0,1)$ such that for some subsequence
$k_j$, we have $t\in T_{k_j}$ for all $j\in\N$.

Thus we can define $\Omega_j:=\{w_{k_j}>t\}$.
\end{proof}

We know the following fact about the extension of sets of finite perimeter: if $\Omega\subset X$ is an open set with $\mathcal H(\partial\Omega)<\infty$ and $E\subset \Omega$ is a $\mu$-measurable set with $P(E,\Omega)<\infty$, then $P(E,X)<\infty$ and in fact
\begin{equation}\label{eq:finite topological boundary result}
P(E,X)\le P(E,\Omega)+C\mathcal H(\partial\Omega),
\end{equation}
see~\cite[Proposition 6.3]{KKST1}. Now we can show a partially more general result.

\begin{corollary}
Let $\Omega\subset X$ be a bounded open set with $P(\Omega,X)<\infty$, and suppose that there exists
$A\subset \partial\Omega$ with $\mathcal H(A)<\infty$ such that
\[
\limsup_{r\to 0^+}\frac{\mu(B(x,r)\setminus\Omega)}{\mu(B(x,r))}>0
\]
for every $x\in \partial\Omega\setminus A$.
Let $E\subset \Omega$ be a $\mu$-measurable set with $P(E,\Omega)<\infty$. Then $P(E,X)<\infty$.
\end{corollary}

\begin{proof}
Take the sequence of sets $\Omega_j\Subset \Omega$ given by Corollary~\ref{cor:approximation finite perimeter from inside}.
By Lebesgue's dominated convergence theorem, we have $\mu(\Om\setminus \Om_j)\to 0$,
and so
by the lower semicontinuity of perimeter and~\eqref{eq:finite topological boundary result},  we have
\begin{align*}
P(E,X)
&\le \liminf_{j\to\infty}P(E\cap\Omega_j,X)\\
&\le \liminf_{j\to\infty}(P(E,\Omega_j)+C\mathcal H(\partial \Omega_j))\\
&\le P(E,\Omega)+C(P(\Omega,X)+\mathcal H(A))<\infty.
\end{align*}
\end{proof}

\begin{example}\label{ex:countably many slits}
Without the requirement of the measure density condition for $X\setminus\Om$ given in the hypothesis 
of the above corollary, the conclusion of the corollary fails. For example, with $\mathbb{D}\subset\R^2$ the unit disk
in $X=\R^2=\mathbb{C}$ centered at $0$, set $\theta_n=\sum_{j=1}^n\frac{\pi}{2^j}$ and let
\[
\Om:=\mathbb{D}\setminus\{z\in\mathbb{C} :\, \Arg(z)=\theta_n\text{ for some }n\in\N\}.
\]
Then $P(\Om,X)=P(\mathbb{D},\R^2)<\infty$. Now with
\[
E=\bigcup_{n\in\N}\{z\in\mathbb{D} :\, \theta_{2n}<\Arg(z)<\theta_{2n+1}\},
\]
we see that $P(E,\Om)=0$, but as $\mathcal{H}(\partial^*E)=\infty$ (note that
$\mathcal H$ is now comparable to the one-dimensional Hausdorff measure), it follows that
$P(E,X)=P(E,\R^2)=\infty$.
\end{example}

\section{Dirichlet problem (T): trace definition}\label{sec:Dirichlet problem trace}


In this section we consider the Dirichlet problem (T) given in Definition~\ref{def:dirichlet}.
We show that the limit of $p$-harmonic functions with boundary data $f$ is a solution to this problem.
%

In the Euclidean setting, it is known that if a bounded domain $\Omega$ has a
Lipschitz boundary, the trace operator $T_+\colon \BV(\Omega)\to L^1(\partial\Omega,\mathcal H)$
is continuous under \emph{strict convergence}, see e.g.~\cite[Theorem 3.88]{AFP}.
In the following proposition we give a generalization of this fact to the metric setting.

\begin{proposition}\label{lemma:use-1}
Let $\Omega\subset X$ be an open set such that the trace operator
$T_+\colon\BV(\Omega)\to L^1(\partial\Omega,\mathcal H)$ is linear and bounded.
Let $u\in \BV(\Om)$, and let $u_k\in\BV(\Omega)$, $k\in\N$, such that
\[
u_k\to u\ \ \textrm{in}\ L^1(\Om)\quad \textrm{and}\quad \Vert Du_k\Vert(\Omega)\to \Vert Du\Vert(\Omega).
\]
Then $T_+ u_k\to T_+u$ in $L^1(\partial\Omega,\mathcal H)$.
\end{proposition}

\begin{proof}
For $t>0$, let
\[
\Om_t:=\{x\in \Om :\, \dist(x, X\setminus\Om)>t\}.
\]
Fix $\eps>0$.
Choose $\eta\in \Lip_c(\Omega)$ such that $0\le \eta\le 1$, $\eta=1$ in $\Omega_{2\eps}$,
$\eta=0$ in $X\setminus \Omega_{\eps}$, and $g_{\eta}\le 1/\eps$.
Let 
\[
v_k:=\eta u +(1-\eta)u_k,\quad k\in\N.
\]
Note that by lower semicontinuity of the total variation with respect to $L^1$-convergence,
\[
\Vert Du\Vert(\Omega_{2\eps})\le \liminf_{k\to\infty}\Vert Du_k\Vert(\Omega_{2\eps}).
\]
Since also $\Vert Du_k\Vert(\Omega)\to \Vert Du\Vert(\Omega)$ by assumption, necessarily
\begin{equation}\label{eq:upper semicontinuity near boundary}
\Vert Du\Vert(\Omega\setminus \Omega_{2\eps})\ge \limsup_{k\to\infty}\Vert Du_k\Vert(\Omega\setminus \Omega_{2\eps}).
\end{equation}
We have $v_k-u=(1-\eta)(u_k-u)$, and so
by the Leibniz rule from~\cite[Lemma~3.2]{HKLS}, and
\eqref{eq:Caccioppoli sets form an algebra},
\begin{align*}
&\limsup_{k\to\infty}\Vert D(v_k-u)\Vert(\Om)\\
     &\qquad \le \limsup_{k\to\infty}\Vert D(u_k-u)\Vert(\Om\setminus\Om_{2\eps})
   +\limsup_{k\to\infty}\int_{\Omega_{\eps}\setminus \Omega_{2\eps}}g_{\eta}|u_k-u|\, d\mu\\
     &\qquad \le \limsup_{k\to\infty}\Vert Du_k\Vert(\Om\setminus\Om_{2\eps})+
     \Vert Du\Vert(\Om\setminus\Om_{2\eps})+0\\
     &\qquad \le 2\Vert Du\Vert(\Om\setminus\Om_{2\eps})
\end{align*}
by~\eqref{eq:upper semicontinuity near boundary}.
Moreover,
$\lim_{k\to\infty}\Vert v_k-u\Vert_{L^1(\Om)}=0$, so in total
\[
\limsup_{k\to\infty}\Vert v_k-u\Vert_{\BV(\Omega)}\le 2 \Vert Du\Vert(\Omega\setminus \Omega_{2\eps}).
\]
Since $T_+$ is assumed to be linear and bounded, for some constant $C_{\Omega}>0$
and for any $v\in\BV(\Omega)$ we have
\[
\int_{\partial\Omega}|T_+v|\,d\mathcal H\le C_{\Omega}\Vert v\Vert_{\BV(\Omega)}.
\]
Note that $T_+ v_k=T_+ u_k$ since $v_k=u_k$ in a neighborhood of $\partial\Omega$. Thus
\begin{align*}
\limsup_{k\to\infty}\int_{\partial\Omega}|T_+u_k-T_+u|\,d\mathcal H
&= \limsup_{k\to\infty}\int_{\partial\Omega}|T_+v_k-T_+u|\,d\mathcal H\\
&\le \limsup_{k\to\infty}C_{\Omega}\Vert v_k-u\Vert_{\BV(\Omega)}\\
&\le 2C_{\Omega}\Vert Du\Vert(\Om\setminus\Om_{2\eps}).
\end{align*}
By letting $\eps\to 0$, we obtain the result.
\end{proof}

\begin{lemma}\label{lem:use1}
Let $\Om\subset X$ be a bounded open set such that $\Om$ satisfies the exterior measure density condition~\eqref{eq:ext-meas-dens-cond}, $\Om$ supports a $(1,1)$-Poincar\'e 
inequality, and there is a 
constant $C\ge 1$ such that whenever
$x\in\partial\Om$ and $0<r<\diam(\Om)$, we have 
\[
 \mu(B(x,r)\cap\Om)\ge \frac{\mu(B(x,r))}{C}.
\]
Assume also that 
for all $x\in\partial\Om$ and $0<r<\diam(\Om)$,
\[
\mathcal H(\Om\cap B(x,r))\le C \frac{\mu(B(x,r))}{r}.
\]
Let $f\in\Lip(X)$ be boundedly supported, and let $u\in \BV(\Omega)$.
Then there exists a sequence
$(\psi_k)\subset\Lip(X)$ converging to $u$ in $L^1(\Omega)$ such that
$\psi_k=f$ in $X\setminus\Omega$ and 
\[
\limsup_{k\rightarrow\infty}\Vert D\psi_k\Vert(\Omega)
\le \Vert Du\Vert(\Omega)+\int_{\partial^*\Omega}|T_+u-f|\, dP_+(\Omega,\cdot).
\]
\end{lemma}

\begin{remark}
Note that some requirement similar to the exterior measure density condition
in the above lemma is 
needed, for without such a requirement we cannot talk about the trace $T_+u$ of a function $u\in \BV(\Om)$. This difficulty is illustrated by the example of the slit disk, 
see~\cite[Example 3.2]{LahSh}.
\end{remark}

\begin{proof}
The assumptions on $\Om$ guarantee that the trace operator
$T_{+}\colon\allowbreak\BV(\Omega)\to L^1(\partial\Omega,\mathcal H)$ is linear and bounded,
see~\cite[Theorem~5.5]{LahSh}.
The assumptions also together imply that 
$\mathcal H(\partial\Om\setminus \partial^*\Om)=0$.

Let $(\eta_{m})\subset \liploc(X)$ such that  $0\le \eta_m\le 1$ on $X$, 
$\eta_{m}=0$ on $X\setminus \Omega$, $\eta_{m}\rightarrow \ch_{\Omega}$ in $L^1(X)$, and
\[
P_+(\Omega,X)=\lim_{m\rightarrow\infty}\int_{\Omega}g_{\eta_m}\, d\mu.
\]
Clearly we have in fact $\eta_m\in\Lip(X)$ for every $m\in\N$.
It is straightforward to check that then also $g_{\eta_m}\,d\mu\to dP_+(\Om,\cdot)$
weakly* in the sense of measures on $X$.
Since $\Om$ supports a $(1,1)$-Poincar\'e inequality, Lipschitz functions are dense in $N^{1,1}(\Om)$, see~\cite[Theorem 5.1]{BB}. It follows that 
there exists a sequence $(\phi_k)\subset \Lip(\Om)$ such that
$\phi_k\to u$ in $L^1(\Om)$ and 
\[
\lim_{k\rightarrow\infty}\int_\Om g_{\phi_k}\, d\mu=\Vert Du\Vert(\Om).
\]
By lower semicontinuity of the total variation with respect to $L^1$-convergence,
necessarily also
\[
\lim_{k\rightarrow\infty}\Vert D\phi_k\Vert(\Omega)=\Vert Du\Vert(\Om).
\]
Now we set 
\[
\psi_{k,m}:=\eta_m\phi_k+(1-\eta_m)f.
\] 
Then $\psi_{k,m}\in\Lip(X)$ and 
\[
\psi_{k,m}\to u\quad\text{in }L^1(\Om)
\]
as $m\to\infty$ and then $k\to \infty$. Furthermore, $\psi_{k,m}=f$ on $X\setminus\Omega$. By the Leibniz rule of~\cite[Lemma 2.18]{BB},
\[
  g_{\psi_{k,m}}\leq g_{\phi_k}\eta_m+g_f(1-\eta_{m})
    +g_{\eta_{m}}|\phi_k-f|.
\]
Here $g_{\psi_{k,m}}$, $g_{\phi_k}$, $g_f$, and $g_{\eta_{m}}$ all denote minimal
$1$-weak upper gradients.
It follows that
\[
\int_\Omega g_{\psi_{k,m}}\, d\mu \le  \int_\Omega g_{\phi_k}\, d\mu+\int_{\Omega}g_f(1-\eta_{m})\, d\mu
  +\int_{\Omega}g_{\eta_{m}}|\phi_k-f|\, d\mu.
\]
As $f$ is a Lipschitz function and $\eta_{m}\rightarrow 1$ in $L^1(\Omega)$, we have
\[
\lim_{m\to \infty} \int_{\Omega}g_f(1-\eta_{m})\, d\mu =0.
\]
Note that the Lipschitz functions $\phi_k$ have Lipschitz extensions to
$X$, which we still denote by $\phi_k$, and that necessarily
$T_+\phi_k=\phi_k$ 
on $\partial \Om$.
Since $g_{\eta_{m}}\,d\mu\rightarrow dP_{+}(\Omega,\cdot)$ weakly* in the sense of
measures, 
\[
\lim_{m\to\infty}\int_{\Omega}|\phi_k-f|g_{\eta_{m}}\, d\mu=\int_{\partial\Omega}|T_+\phi_k-f|\, dP_+(\Omega,\cdot).
\]
It follows from Lemma~\ref{lemma:use-1} that $T_+\phi_k\to T_+u$ in $L^1(\partial\Omega,\mathcal H)$, and then by~\eqref{eq:def of theta} and Theorem~\ref{thm:P+vsP},
also $T_+\phi_k\to T_+u$ in $L^1(\partial\Omega,P_+(\Om,\cdot))$. Thus,
recalling also that $\mathcal H(\partial\Om\setminus\partial^*\Om)=0$,
\[
\lim_{k\to\infty}\lim_{m\to\infty}\int_{\Omega}|\phi_k-f|g_{\eta_{m}}\, d\mu
=\int_{\partial^*\Omega}|T_+u-f|\, dP_+(\Omega,\cdot),
\]
and now we can choose a diagonal sequence $\{\psi_{k,m_k}\}_k$ to satisfy the conclusion of the lemma.
%
%
\end{proof}

In what follows, we denote by $T_-u$
the outer trace (if it exists) of a $\BV$ function $u\in \BV(X)$, namely, 
$T_-u$ is the interior trace of $u$ considered with respect to $X\setminus\Om$ as
given in Definition~\ref{def:Madayan-traces}. We will only need the following
proposition for the
case where $u=f$ on $X\setminus\Om$ for some Lipschitz function $f$; in this case, we always have
$T_-u=f$ on $\partial\Om\setminus N_{X\setminus\Om}$, in particular, $T_-u=f$ on $\partial^*\Om$.

\begin{proposition}\label{prop:representation of variation with boundary term}
Let $\Omega\subset X$ be a $\mu$-measurable set with $P(\Omega,X)<\infty$ and let $u\in\BV(X)$ such that for $\mathcal H$-almost every $x\in\partial^*\Omega$, $T_+ u(x)$ and $T_- u(x)$ exist. Then
\[
\Vert Du\Vert(X)=\Vert Du\Vert(X\setminus\partial^*\Omega)+\int_{\partial^*\Omega}|T_+u-T_-u|\,d P(\Omega,\cdot).
\]
\end{proposition}

\begin{proof}
We only need to prove that
\[
\Vert Du\Vert(\partial^*\Omega)=\int_{\partial^*\Omega}|T_+u-T_-u|\,d P(\Omega,\cdot).
\]
By~\cite[Theorem 5.3]{AMP}, we have $\Vert Du\Vert^c(\partial^*\Om)=0$, and then by
the decomposition~\eqref{eq:decomposition},
\[
\Vert Du\Vert(\partial^*\Omega)=\int_{\partial^*\Omega}\int_{u^{\wedge}(x)}^{u^{\vee}(x)}\theta_{\{u>t\}}(x)\,dt\,d\mathcal H(x).
\]
It is fairly easy to check that $\{u^{\wedge}(x),u^{\vee}(x)\}=\{T_-u(x),T_+u(x)\}$, whenever both traces exist.  This is also proved in~\cite[Proposition 5.8(v)]{HKLL}. Suppose that $u^{\wedge}(x)=T_-u(x)$ and
$u^{\vee}(x)=T_+u(x)$, the other case being analogous. In the proof of~\cite[Proposition 5.8(v)]{HKLL} it is also shown that
\[
\lim_{r\to 0^+}\frac{\mu(B(x,r)\cap(\{u>t\}\Delta \Omega))}{\mu(B(x,r))}=0
\]
for all $t\in (u^{\wedge}(x),u^{\vee}(x))$. We also have $x\in\partial^*\{u>t\}$ for all $t\in (u^{\wedge}(x),u^{\vee}(x))$. 
According to~\cite[Proposition 6.2]{AMP}, we have $\theta_{\{u>t\}}(x)=\theta_{\Omega}(x)$ for $\mathcal H$-almost every such $x$.
Hence we have
\begin{align*}
 &\int_{\partial^* \Omega}\int_{u^{\wedge}(x)}^{u^{\vee}(x)}\theta_{\{u>t\}}(x)\,dt\,d\mathcal H(x)\\
 &=\int_{\partial^* \Omega}\int_{-\infty}^{\infty}\ch_{\{(u^{\wedge}(x),u^{\vee}(x))\}}(t)\theta_{\{u>t\}}(x)\,dt\,d\mathcal H(x)\\
 & =\int_{-\infty}^{\infty}\int_{\partial^* \Omega}\ch_{\{(-\infty,t)\}}(u^{\wedge}(x))\ch_{\{(t,\infty)\}}(u^{\vee}(x))
    \theta_{\{u>t\}}(x)\,d\mathcal H(x)\,dt\\
 & =\int_{-\infty}^{\infty}\int_{\partial^* \Omega}\ch_{\{(-\infty,t)\}}(u^{\wedge}(x))\ch_{\{(t,\infty)\}}(u^{\vee}(x))
     \theta_{\Omega}(x)\,d\mathcal H(x)\,dt\\
 &=\int_{\partial^* \Omega}\int_{-\infty}^{\infty}\ch_{\{(u^{\wedge}(x),u^{\vee}(x))\}}(t)\,dt\,\theta_{\Omega}(x)\,d\mathcal H(x)\\
 &=\int_{\partial^* \Omega}(u^{\vee}(x)-u^{\wedge}(x))\theta_{\Omega}(x)\,d\mathcal H(x)\\
  &=\int_{\partial^* \Omega}(u^{\vee}-u^{\wedge})\,dP(\Omega,\cdot)\\
   &=\int_{\partial^* \Omega}|T_+u-T_-u|\,dP(\Omega,\cdot).
\end{align*}
\end{proof}

For $\mu$-measurable $\Om\subset X$ and any $\kappa>0$, define the weighted measure
\begin{equation}\label{eq:definition of kappa mu}
d \mu_{\kappa}:=(\ch_{\Omega}+\kappa\,\ch_{X\setminus \Omega})\,d \mu.
\end{equation}
Consider then the space $(X,d,\mu_{\kappa})$. It is easy to show that this is still a complete metric 
space such that $\mu_{\kappa}$ is doubling and supports a $(1,1)$-Poincar\'e inequality.
We use the subscript $\kappa$ to signify that a perimeter or some other quantity is
taken with respect to the measure $\mu_{\kappa}$.

\begin{theorem}\label{thm:dirichletT}
Let $\Om\subset X$ be a nonempty bounded open set of finite perimeter such that
$\capa_1(X\setminus\Om)>0$,
$\Om$ satisfies the exterior measure density condition~\eqref{eq:ext-meas-dens-cond},
and $\Om$ supports a $(1,1)$-Poincar\'e 
inequality. Suppose also that there is a 
constant $C\ge 1$ such that whenever
$x\in\partial\Om$ and $0<r<\diam(\Om)$, we have 
\[
 \mu(B(x,r)\cap\Om)\ge \frac{\mu(B(x,r))}{C}.
\]
Finally, assume that 
for all $x\in\partial\Om$ and $0<r<\diam(\Om)$,
\[
\mathcal H(\Om\cap B(x,r))\le C\, \frac{\mu(B(x,r))}{r}.
\]
Let $f\in \Lip(X)$ be boundedly supported. 
For each $p>1$ let $u_p$ be a \p-harmonic function in $\Om$ such that $u_p|_{X\setminus\Om}=f$. Suppose 
that $(u_p)_{p>1}$ is a sequence of such \p-harmonic functions and that $u_p \to u$ in $L^1(\Om)$ as 
$p\to 1^+$. Then $u$ is a solution to the minimization problem $(T)$ of Definition~\ref{def:dirichlet}.
\end{theorem}

\begin{proof}[Beginning of the proof of Theorem~\ref{thm:dirichletT}]
Note that by combining the exterior measure density condition
\eqref{eq:ext-meas-dens-cond} and~\eqref{eq:density of E},
we obtain that
\begin{equation}\label{eq:exterior measure density condition in theorem consequence}
\liminf_{r\to 0^+}\frac{\mu(B(x,r)\setminus \Omega)}{\mu(B(x,r))}\ge\gamma\qquad\textrm{for }\mathcal H\textrm{-a.e. }x\in\partial\Omega.
\end{equation}
Note also again that the assumptions on $\Om$ guarantee that the trace operator
$T_{+}\colon\BV(\Omega)\to L^1(\partial\Omega,\mathcal H)$ is linear and bounded,
see~\cite[Theorem~5.5]{LahSh}.

Let $v\in \BV(\Omega)$. By combining \eqref{eq:def of theta} and
Theorem \ref{thm:P+vsP}, we know that $P_+(\Om,\cdot)$ is concentrated on $\partial^*\Om$.
Thus we need to show that
\[
\Vert Dv\Vert(\Om)+\int_{\partial^*\Om}|T_+v-f|\, dP_+(\Omega,\cdot)\le
\Vert Dv\Vert(\Om)+\int_{\partial^*\Om}|T_+v-f|\, dP_+(\Omega,\cdot).
\]
By Lemma~\ref{lem:use1},
there is a sequence $(\psi_k)\subset\Lip(X)$ with $\psi_k=f$ in $X\setminus\Om$ such that
\[
\limsup_{k\to\infty} \int_\Om g_{\psi_k}\, d\mu\le \Vert Dv\Vert(\Om)+\int_{\partial^*\Om}|T_+v-f|\, dP_+(\Omega,\cdot).
\]
Observe that each $\psi_k$ can act as a test function for testing the $p$-harmonicity of $u_p$. Therefore by~\eqref{eq:comparing variation and upper gradients}
\[
\Vert Du_p\Vert(\Om)\le \mu(\Omega)^{1-1/p}\left(\int_\Om g_{u_p}^p\, d\mu\right)^{1/p}
\le \mu(\Omega)^{1-1/p}\left(\int_\Om g_{\psi_k}^p\, d\mu\right)^{1/p}.
\]
Letting $p\to 1^+$, we see that
\[
 \limsup_{p\to 1^+}\Vert Du_p\Vert(\Om)\le \int_\Om g_{\psi_k}\, d\mu.
\]
Therefore by now letting $k\to\infty$, we have
\begin{equation}\label{eq:almost-close}
\limsup_{p\to 1^+}\Vert Du_p\Vert(\Om)\le
\Vert Dv\Vert(\Om)+\int_{\partial^*\Om}|T_+v-f|\, dP_+(\Omega,\cdot).
\end{equation}
Thus we need to prove that
\begin{equation}\label{eq:bdry-jump}
\limsup_{p\to 1^+}\Vert Du_p\Vert(\Om)\ge \Vert Du\Vert(\Om)+\int_{\partial^*\Om}|T_+u-f|\, dP_+(\Omega,\cdot)
\end{equation}
in order to complete the proof.

Recall the definitions of $O_\Om$ and $I_\Om$ from Definition~\ref{def:int-ext-bdy-meas}.
By the exterior measure density condition~\eqref{eq:ext-meas-dens-cond}, we know that
$\mathcal{H}(\partial\Om\cap I_{\Om})=0$.
Recall the definition of $(X,d,\mu_{\kappa})$ from~\eqref{eq:definition of kappa mu}.
We note that $\Vert D_{\kappa}u\Vert$ is absolutely continuous with respect to $\mathcal H$,
which follows from the $\BV$ coarea formula~\eqref{eq:coarea} and~\eqref{eq:def of theta}.
Thus $\Vert D_{\kappa}u\Vert(I_{\Omega}\setminus \Omega)=0$.
Note also that since $u=f$ on $X\setminus\Omega$,
$\partial^*\{u-f>t\}\cap O_\Om=\emptyset$ for all $t\in\R$, and so by the coarea formula
and~\eqref{eq:def of theta}
\begin{align*}
\Vert D_{\kappa}(u-f)\Vert(O_\Om)
&=\int_{-\infty}^{\infty}P_{\kappa}(\partial^*\{u-f>t\}, O_\Om)\,dt\\
&\le C(C_d,\kappa)\int_{-\infty}^{\infty}\mathcal H_{\kappa}(\partial^*\{u-f>t\}\cap O_\Om)\,dt\\
&=0.
\end{align*}
Then by the lower semicontinuity of the total variation and 
Proposition~\ref{prop:representation of variation with boundary term}, we have
\begin{align*}
&\liminf_{p\to 1^+}\Vert D_{\kappa}u_p\Vert(X)
\ge \Vert D_{\kappa}u\Vert(X)\\
&\qquad\qquad=\Vert D_{\kappa}u\Vert(I_\Om)+\Vert D_{\kappa}u\Vert(O_{\Omega})
+\Vert D_{\kappa}u\Vert(\partial^*\Om)\\
&\qquad\qquad=\Vert D_{\kappa}u\Vert(\Omega)+\Vert D_{\kappa}u\Vert(O_{\Omega})+\int_{\partial^*\Omega}|T_+u-f|\,d P_{\kappa}(\Omega,\cdot)\\
&\qquad\qquad=\Vert Du\Vert(\Omega)+\Vert D_{\kappa}f\Vert(O_{\Omega})+
\int_{\partial^*\Omega}|T_+u-f|\,d P_{\kappa}(\Omega,\cdot).
\end{align*}
Similarly, on the left-hand side we have
\begin{align*}
\Vert D_{\kappa}u_p\Vert(X)
&=\Vert Du_p\Vert(\Omega)+\Vert D_{\kappa}u_p\Vert(\partial^*\Omega)+\Vert D_{\kappa}u_p\Vert(O_{\Omega})\\
&=\Vert Du_p\Vert(\Omega)+\Vert D_{\kappa}f\Vert(O_{\Omega});
\end{align*}
note that $\Vert D_{\kappa}u_p\Vert(\partial^*\Omega)=0$ since $\mu(\partial^*\Omega)=0$. In total, we have
\begin{equation}\label{eq:p goes to one for kappa}
\liminf_{p\to 1^+}\Vert Du_p\Vert(\Om)\ge \Vert Du\Vert(\Om)+\int_{\partial^*\Omega}|T_+u-f|\, dP_{\kappa}(\Omega,\cdot).
\end{equation}
 The inequality~\eqref{eq:bdry-jump} will follow from the above inequality if we know that
\begin{equation}\label{eq:p-1-kappa-infty}
\lim_{\kappa\to\infty}\int_{\partial^*\Omega}|T_+u-f|\, dP_{\kappa}(\Omega,\cdot)
=\int_{\partial^*\Omega}|T_+u-f|\, dP_+(\Omega,\cdot).
\end{equation}
This is the focus of the rest of this section, and we will complete the proof at the end of the section.
\end{proof}

We will need the following approximation of a set of finite perimeter by "regular" sets.
This is inspired by a similar result in~\cite{AMG}, but note that we use a somewhat different, ``two-sided" 
definition of the Minkowski content, as given in~\eqref{eq:definition of Minkowski content}.
First recall that by~\cite[Proposition 3.5]{KoLa} (which is based on~\cite{BH}) the following coarea inequality holds:
for any $w\in\Lip_c(X)$,
\[
\int_{-\infty}^{\infty}\nu^+(\partial\{w>t\})\,dt\le \int_{X}\Lip w\,d\nu,
\]
where $\nu$ is any positive Radon measure. From this it follows in a straightforward manner
that
for any $w\in\liploc(X)$,
\begin{equation}\label{eq:coarea with Minkowski content}
\int_{-\infty}^{\infty}\mu^+(\partial\{w>t\})\,dt\le \int_{X}\Lip w\,d\mu.
\end{equation}

\begin{lemma}\label{lem:approximation by regular sets}
Let $E\subset X$ be a set of finite perimeter. Fix $0<\delta<1$.
Then there exists a sequence of open sets of finite perimeter $E_i\subset X$ with
$\ch_{E_i}-\ch_{E}\to 0$ in $L^1(X)$,
$\mu(\partial E_i)=0$ for each $i\in\N$,
\[
\limsup_{i\to\infty}P(E_i,X)\le (1-\delta)^{-1}P(E,X),
\]
and
\[
\limsup_{i\to\infty}\mu^+(\partial E_i)\le\frac{C P(E,X)}{\delta}. 
\]
\end{lemma}

\begin{proof}
By Lemma~\ref{lem:L1 loc and L1 convergence}, we can pick a sequence
$(v_i)\subset\liploc(X)$
with $v_i- \ch_E\to 0$ in $L^1(X)$ and $\int_X g_{v_i}\,d\mu\to P(E,X)$,
where each $g_{v_i}$ is the minimal $1$-weak upper gradient of $v_i$. We may also choose the functions so that $v_{i}\geq 0$.
Furthermore, $\Lip v_i\le C g_{v_i}$ $\mu$-almost everywhere,
see~\cite[Proposition 4.26]{Che} or~\cite[Proposition 13.5.2]{HKST}. 
According to the coarea formula for $\BV$ functions, see~\eqref{eq:coarea},
for every $i\in\N$ we have
\[
\int_0^1 P(\{v_i>t\},X)\,dt\le \int_X g_{v_i}\,d\mu.
\]
Now by Chebyshev's inequality,
\[
\mathcal L^1\left(\left\{t\in [0,1]:\, P(\{v_i>t\},X)>(1-\delta)^{-1}\int_X g_{v_i}\,d\mu \right\}\right)
\le 1-\delta;
\]
note that this holds also if $\int_X g_{v_i}\,d\mu=0$, as then $P(\{v_i>t\},X)=0$ for a.e. $t\in [0,1]$.
Therefore there is a measurable set
$A_i\subset [\delta/4,1-\delta/4]$ with $\mathcal L^1(A_i)\ge \delta/2$ and 
\[
P(\{v_i>t\},X)\le \frac{1}{1-\delta}\int_X g_{v_i}\,d\mu
\]
for all $t\in A_i$.
Moreover, since the sets $\partial\{v_i>t\}\subset \{v_i=t\}$ are disjoint for distinct values of $t$, 
we have $\mu(\partial \{v_i>t\})=0$ for a.e. $t\in [0,1]$.
By the version of the coarea formula found in~\eqref{eq:coarea with Minkowski content}, we have
\[
\int_0^1\mu^+(\partial\{v_i>t\})\,dt\le \int_X \Lip v_i\,d\mu.
\]
Thus for each $i\in\N$, there exists $t_i\in A_i$ with
\[
\mu^+(\partial\{v_i>t_i\})\le \frac{2}{\delta}\int_X \Lip v_i\,d\mu
\le \frac{C}{\delta}\int_X g_{v_i}\,d\mu
\]
and $\mu(\partial \{v_i>t_i\})=0$.
Define $E_i:=\{v_i>t_i\}$. Then
\[
\limsup_{i\to\infty}P(E_i,X)\le (1-\delta)^{-1}P(E,X),
\]
and
\[
\limsup_{i\to\infty}
\mu^+(\partial E_i)\le \frac{C}{\delta} P(E,X). 
\]
Now we need to show that $\ch_{\{v_i>t_i\}}- \ch_E\to 0$ in $L^1(X)$.
Note that for any $t\in [\delta/4,1-\delta/4]$, 
for any $x$ such that 
\[
x\in \{v_i>t\}\cap E\quad\text{ or  }\quad x\in X\setminus (\{v_i>t\}\cup E),
\]
we have 
\[
|\ch_{\{v_i>t\}}(x)-\ch_{E}(x)|=0\le 4|v_i(x)-\ch_E(x)|/\delta. 
\]
If $x\in\{v_i>t\}$ but $x\not\in E$,
then 
\[
|\ch_{\{v_i>t\}}(x)-\ch_{E}(x)|=1\le v_i(x)/t\le 4|v_i(x)-\ch_E(x)|/\delta.
\]
On the other hand, if
$x\not\in\{v_i>t\}$ but $x\in E$, then
$v_{i}(x)\leq t\le 1- \delta/4$
and it follows that
\[
|\ch_{\{v_i>t\}}(x)-\ch_{E}(x)|=1\le 4|v_i(x)-\ch_E(x)|/\delta.
\]
Thus
\[
\int_X |\ch_{\{v_i>t_i\}}-\ch_{E}|\,d\mu\le \frac{4}{\delta}\int_X |v_i-\ch_{E}|\,d\mu\to 0
\]
so that $\ch_{\{v_i>t_i\}}- \ch_E\to 0$ in $L^1(X)$ (even though $t_i$ depends on $i$).
\end{proof}

The reason for utilizing the Minkowski content is that it scales nicely according to the
parameter $\kappa$ in $\mu_{\kappa}$, in the following sense.

\begin{lemma}\label{lem:kappa and Minkowski content}
Let $\Omega\subset X$ be $\mu$-measurable,
let $A\subset X\setminus \Omega$, let $\beta>0$, and suppose that there is some $R>0$ for which
\begin{equation}\label{eq:exterior measure density condition with uniform constant for small radii}
\inf_{0<r<R}\frac{\mu(B(x,r)\setminus \Omega)}{\mu(B(x,r))}\ge \frac{\beta}{2}
\end{equation}
for every $x\in A$.
Then we have
\[
C\mu^+_{\kappa}(A)\ge  \kappa \beta \mu^+(A).
\]
\end{lemma}
\begin{proof}
For any $x\in A$ and radii $r\in (0,R)$,
\[
\mu_{\kappa}(B(x,r))\ge \mu_{\kappa}(B(x,r)\setminus\Omega)=\kappa\mu(B(x,r)\setminus\Omega)
\ge \frac{\kappa\beta}{2}\mu(B(x,r)).
\]
Fix $0<r<R/5$ and consider the collection of balls $\{B(x,r)\}_{x\in A}$. By the 5-covering theorem we can pick a countable collection of disjoint balls
$B(x_j,r)$ such that the balls $B(x_j,5r)$ cover $\bigcup_{x\in A}B(x,r)$. We have
\begin{align*}
\frac{\mu\left(\bigcup_{x\in A}B(x,r)\right)}{2r}
\le \sum_{j\in\N}\frac{\mu(B(x_j,5r))}{2r}
&\le C_d^3\sum_{j\in\N}\frac{\mu(B(x_j,r))}{2r}\\
&\le \frac{2C_d^3}{\kappa\beta}\sum_{j\in\N}\frac{\mu_{\kappa}(B(x_j,r))}{2r}\\
&\le \frac{2C_d^3}{\kappa\beta}\frac{\mu_{\kappa}\left(\bigcup_{x\in A}B(x,r)\right)}{2r}.
\end{align*}
By taking the limit infimum as $r\to 0$ on both sides, we obtain
\[
\mu^+(A)\le \frac{2C_d^3}{\kappa\beta} \mu_{\kappa}^+(A).
\]
\end{proof}

Moreover, we have the following simple estimate for the Minkowski content and Hausdorff measure that 
we will need in the proof of Proposition~\ref{prop:comparison of P plus and P kappa}. 
The estimate can be proved by a simple covering argument, see~\cite[Proposition~3.12]{KoLa}.

\begin{lemma}\label{lem: Hausdorff and Minkowski}
For any $A\subset X$, we have $\mathcal H(A)\le C_d^3\mu^+(A)$.
\end{lemma}

It is less clear 
that this estimate would hold if we 
used a ``one-sided" definition of the Minkowski content, as for example in~\cite{AMG}.

\begin{lemma}\label{lem:lower semicontinuity for P+}
Let $\Omega\subset X$ be an open set. If $K_i\subset \Omega$, $i\in\N$, are compact sets
with $\ch_{K_i}-\ch_\Om\to 0$ in $L^1(X)$,
then
\[
P_+(\Omega,X)\le\liminf_{i\to\infty}P(K_i,X).
\]
\end{lemma}

\begin{proof}
By~\cite[Lemma 2.6]{HaSh}, for each $i\in\N$ we can find a function $v_i\in\Lip_c(\Om)$
such that $\Vert v_i-\ch_{\Om}\Vert_{L^1(\Omega)}<1/i$ and
\[
\int_\Om g_{v_i}\,d\mu\le P(K_i,\Om)+1/i.
\]
The conclusion follows by the definition of $P_+(\Om,\cdot)$.
\end{proof}

\begin{proposition}\label{prop:comparison of P plus and P kappa}
Let $\Omega\subset X$ be a bounded open set with $P(\Om,X)<\infty$, and assume that for some constant $\beta>0$, we have
\begin{equation}\label{eq:exterior measure density condition with uniform constant}
\liminf_{r\to 0^+}\frac{\mu(B(x,r)\setminus \Omega)}{\mu(B(x,r))}\ge \beta
\end{equation}
for $\mathcal H$-a.e. $x\in\partial\Omega$.
Then we have
\[
P_{+}(\Omega,X)= \lim_{\kappa\to\infty}P_{\kappa}(\Omega,X).
\]
\end{proposition}
\begin{proof}
By Theorem~\ref{thm:P+vsP} we have $P_+(\Om,X)<\infty$.
Note that for any $\kappa>0$ we have $P_{+}(\Omega,X)\ge P_{\kappa}(\Omega,X)$, 
so only the other inequality needs to be proved.
Fix  $0<\delta<1$, and fix $\kappa>0$. By Lemma~\ref{lem:approximation by regular sets} 
we can find a sequence of open sets $\Omega_i\subset X$ of finite perimeter such that
$\ch_{\Omega_i}\to\ch_{\Omega}$ in $L^1(X)$,
$\mu(\partial\Omega_i)=0$ for all $i\in\N$,
\begin{equation}\label{eq:perimeters of Omega i and Omega}
\limsup_{i\to\infty}P_{\kappa}(\Omega_i,X)\le (1-\delta)^{-1} P_{\kappa}(\Omega,X),
\end{equation}
and
\begin{equation}\label{eq:Minkowski contents of Omega i}
\limsup_{i\to\infty}\mu^+_{\kappa}(\partial\Omega_i)\le \frac{C}{\delta} P_{\kappa}(\Omega,X)\le \frac{C}{\delta} P_{+}(\Omega,X).
\end{equation}
In the following, we will repeatedly use the measure property and the 
subadditivity property~\eqref{eq:Caccioppoli sets form an algebra} of sets of finite perimeter.
Since $\ch_{\Om_i}\to\ch_{\Om}$ in $L^1(X)$, it follows that
$\ch_{\Om\cup\Om_i}\to\ch_\Om$ in $L^1(X)$ as well.
By the lower semicontinuity of perimeter
and the fact that the perimeter of a set is
concentrated on its measure theoretic boundary,
we estimate 
\begin{align*}
P(\Omega,X)
&\le \liminf_{i\to\infty} P(\Omega\cup\Omega_i,X)\\
&= \liminf_{i\to\infty} P(\Omega\cup\Omega_i,X\setminus (I_{\Omega_i}\cup I_{\Omega}))\\
&\le \liminf_{i\to\infty} \Big(P(\Omega,X\setminus (I_{\Omega_i}\cup I_{\Omega}))
    +P(\Omega_i,X\setminus (I_{\Omega_i}\cup I_{\Omega}))\Big)\\
&\le \liminf_{i\to\infty} \Big(P(\Omega,X\setminus I_{\Omega_i})+P(\Omega_i,X\setminus  I_{\Omega})\Big).
\end{align*}
It follows that
\begin{equation}\label{eq:lower semicontinuity estimate}
P(\Omega,I_{\Omega_i})\le P(\Omega_i,X\setminus \Omega)+\eps_i,
\end{equation}
where $\eps_i\to 0$ as $i\to\infty$.
For any sets $A,B\subset X$, we have
\[
\partial^*(A\cap B)\subset (\partial^*A\cap \partial^*B)\cup (\partial^*A\cap I_B)
\cup (\partial^*B\cap I_A).
\]
Thus we have
\begin{align*}
\partial^*(\Omega_i\setminus \Omega)
&\subset (\partial^*\Omega_i\cap \partial^*\Omega)
\cup (\partial^*\Omega_i\cap O_{\Omega})\cup (\partial^*\Omega\cap I_{\Omega_i})\\
&\subset(\partial^*\Omega_i\setminus \Omega)\cup (\partial^*\Omega\cap I_{\Omega_i}).
\end{align*}
By~\eqref{eq:Caccioppoli sets form an algebra}, $P(\Omega_i\setminus\Omega,X)<\infty$,
and then by using~\eqref{eq:def of theta}, we obtain
\begin{align*}
P(\Omega_i\setminus\Omega, X)
&\le C\mathcal H(\partial^*(\Omega_i\setminus \Omega))\\
&\le C\mathcal H(\partial^*\Omega_i\setminus \Omega)+C\mathcal H(\partial^*\Omega\cap I_{\Omega_i})\\
&\le CP(\Omega_i,X\setminus\Omega)+CP(\Omega,I_{\Omega_i}).
\end{align*}
Combining this with~\eqref{eq:lower semicontinuity estimate}, we obtain that for all $i\in\N$
\begin{equation}\label{eq:first estimate for difference of Omegas}
P(\Omega_i\setminus\Omega,X)\le C P(\Omega_i,X\setminus \Omega)+C\eps_i.
\end{equation}
Note that $\mathcal{H}\vert_{\partial\Om_i}$ is a Borel measure of finite mass,
since by Lemma~\ref{lem: Hausdorff and Minkowski},
\[
\mathcal H(\partial \Omega_i)\le C\mu^+(\partial \Omega_i)\le C\mu_{\kappa}^+(\partial \Omega_i)<\infty.
\]
Note that for any fixed $r>0$,
the map 
\[x\mapsto \frac{\mu(B(x,r)\setminus\Om)}{\mu(B(x,r))}
\]
is a Borel map as the ratio of two lower semicontinuous functions. Hence for each $\tau>0$
the function $f_\tau$ given by 
\[
f_\tau(x)=\inf_{r\in \Q\cap(0,\tau)}\frac{\mu(B(x,r)\setminus\Om)}{\mu(B(x,r))}
\]
is also Borel measurable,
and so is
\[
f_\infty(x):=\liminf_{r\to0^+}\frac{\mu(B(x,r)\setminus\Om)}{\mu(B(x,r))}=\lim_{\tau\to 0^+}f_\tau(x).
\]
So for each $i\in\N$, by Egorov's theorem we can choose a set
$A_i\subset \partial \Omega_i$ with $\mathcal H(\partial \Omega_i\setminus A_i)<\eps_i$
and such that
$f_\tau\to f_\infty$ uniformly in $A_i$. Thus~\eqref{eq:exterior measure density condition with uniform constant for small radii} is satisfied for $A=A_i\setminus\Omega$ and some $R>0$.
By Lemma~\ref{lem: Hausdorff and Minkowski} and Lemma~\ref{lem:kappa and Minkowski content}, we get
\begin{equation}\label{eq:estimate for kappa Hausdorff measure}
\begin{split}
\mathcal H(\partial\Omega_i\setminus\Omega)
&\le \mathcal H(A_i\setminus\Omega)+\eps_i\\
&\le C\mu^+(A_i\setminus\Omega)+\eps_i\\
&\le \frac{C}{\kappa\beta}\mu_{\kappa}^+(A_i\setminus\Omega)+\eps_i\\
&\le \frac{C}{\kappa\beta}\mu_{\kappa}^+(\partial\Omega_i)+\eps_i.
\end{split}
\end{equation}
Then by~\eqref{eq:def of theta}
\[
P(\Omega_i,X\setminus \Omega)
\le C\mathcal H(\partial\Omega_i\setminus\Omega)
\le \frac{C}{\kappa\beta}\mu_{\kappa}^+(\partial\Omega_i)+C\eps_i,
\]
so by combining with~\eqref{eq:first estimate for difference of Omegas}, we have
\[
P(\Omega_i\setminus\Omega,X)\le \frac{C}{\kappa\beta}\mu_{\kappa}^+(\partial\Omega_i)
+C\eps_i.
\]
Recall that $\mu(\overline{\Omega_i}\setminus \Omega_i)=\mu(\partial\Omega_i)=0$
for all $i\in\N$. Thus by~\eqref{eq:Minkowski contents of Omega i},
\begin{equation}\label{eq:estimate for protruding part}
\begin{split}
\limsup_{i\to\infty}P(\overline{\Omega_i}\setminus\Omega,X)
&=\limsup_{i\to\infty}P(\Omega_i\setminus\Omega,X)\\
&\le \limsup_{i\to\infty}\frac{C}{\kappa\beta}\mu_{\kappa}^+(\partial\Omega_i)\\
&\le \frac{C}{\kappa\beta\delta}P_{+}(\Omega,X).
\end{split}
\end{equation}
For $A\subset X$, we set 
\[
A^{\beta}:=\left\{x\in X:\, \limsup_{r\to 0^+}\frac{\mu(B(x,r)\cap A)}{\mu(B(x,r))}\ge\beta\right\}.
\]
Let us denote by $D\subset \partial \Omega$ the $\mathcal H$-negligible set where 
\eqref{eq:exterior measure density condition with uniform constant} fails.
Note that $\mu(X)>0$, and so we can assume that $\mu(\Om_i\setminus \Om)< \mu(X)/2$
for all $i\in\N$.
Now by the boxing 
inequality, see~\cite[Remark 3.3(1)]{KKST2},
we can find a collection of balls $\{B(x_j^i,r_j^i)\}_{j\in\N}$ covering
$(\overline{\Omega_i}\setminus\Omega)^{\beta}$ such that the balls $B(x_j^i,r_j^i/5)$
are disjoint,
\begin{equation}\label{eq:measure of balls from boxing inequality}
\frac{\mu(B(x_j^i,r_j^i/5)\cap \overline{\Omega_i}\setminus\Omega)}{\mu(B(x_j^i,r_j^i/5))}
\ge \frac{\beta}{C},
\end{equation}
and
\[
\sum_{j\in\N}\frac{\mu(B(x_j^i,r_j^i))}{r_j^i}
\le C \beta^{-1}P(\overline{\Omega_i}\setminus\Omega,X).
\]
Note that in~\cite{KKST2} it is assumed that $\mu(X)=\infty$,
but the condition $\mu(\Om_i\setminus \Om)< \mu(X)/2$ is sufficient for the proof to
work.
Then by~\eqref{eq:estimate for protruding part},
\begin{equation}\label{eq:first estimate for balls}
\limsup_{i\to\infty}\sum_{j\in\N}\frac{\mu(B(x_j^i,r_j^i))}{r_j^i}\le \frac{C}{\kappa\beta^2\delta}P_{+}(\Omega,X).
\end{equation}
Note that if $x\in I_{\Omega_i}\setminus( \Omega \cup D)$, then
$x\in (\overline{\Omega_i}\setminus\Omega)^{\beta}$.
Thus we have
\[
(\overline{\Omega_i}\setminus\Omega)\setminus (\overline{\Omega_i}\setminus\Omega)^{\beta} 
  \subset D\cup \partial\Omega_i\setminus \Omega.
\]
But by~\eqref{eq:estimate for kappa Hausdorff measure},
\[
\mathcal H(D\cup \partial\Omega_i\setminus \Om)=\mathcal H(\partial\Omega_i\setminus\Omega)
  \le \frac{C}{\kappa\beta}\mu_{\kappa}^+(\partial\Omega_i)+\eps_i.
\]
Thus we can pick another collection $\{B(y_k^i,s_k^i)\}_{k\in\N}$ of balls covering
$(\overline{\Omega_i}\setminus\Omega)\setminus (\overline{\Omega_i}\setminus\Omega)^{\beta}$ with $s_k^i\le 1/i$ and
\[
\sum_{k\in\N}\frac{\mu(B(y_k^i,s_k^i))}{s_k^i}
\le \frac{C}{\kappa\beta}\mu_{\kappa}^+(\partial\Omega_i)+2\eps_i,
\]
and so by~\eqref{eq:Minkowski contents of Omega i},
\begin{equation}\label{eq:second estimate for balls}
\limsup_{i\to\infty}\sum_{k\in\N}\frac{\mu(B(y_k^i,s_k^i))}{s_k^i}
\le \frac{C}{\kappa\beta\delta}P_{+}(\Omega,X).
\end{equation}
Note that the collections $\{B(x_j^i,r_j^i)\}_{j\in\N}$ and $\{B(y_k^i,s_k^i)\}_{k\in\N}$
together cover all of $\overline{\Omega_i}\setminus\Omega$.
By~\cite[Lemma 6.2]{KKST1} we can pick radii $\widetilde{r}_j^i\in [r_j^i,2r_j^i]$
such that
\[
\frac{1}{C}P(B(x_j^i,\widetilde{r}_j^i),X)\le \frac{\mu(B(x_j^i,\widetilde{r}_j^i))}{\widetilde{r}_j^i}\le C_d\frac{\mu(B(x_j^i,r_j^i))}{r_j^i},
\]
and similarly we find radii $\widetilde{s}_k^i\in [s_k^i,2s_k^i]$.
Define for each $i\in\N$
\[
K_i:=\overline{\Omega_i}\setminus 
  \left(\bigcup_{j\in\N}B(x_j^i,\widetilde{r}_j^i)\cup \bigcup_{k\in\N}B(y_k^i,\widetilde{s}_k^i)\right).
\]
Note that these are closed sets contained in $\Omega$, and thus compact.
By~\eqref{eq:first estimate for balls} and~\eqref{eq:second estimate for balls}, we have
\begin{align*}
&\limsup_{i\to\infty}P(K_i,X)\\
&\quad\le \limsup_{i\to\infty} \left( P(\overline{\Omega}_i,X)+\sum_{j\in\N}P(B(x_j^i,\widetilde{r}_j^i),X)
   +\sum_{k\in\N}P(B(y_k^i,\widetilde{s}_k^i),X)\right)\\
&\quad  \le \limsup_{i\to\infty} P_{\kappa}(\overline{\Omega}_i,X)+\frac{C}{\kappa
\beta^2\delta}P_{+}(\Omega,X)+\frac{C}{\kappa\beta\delta}P_{+}(\Omega,X)\\
&\quad  = \limsup_{i\to\infty} P_{\kappa}(\Omega_i,X)+\frac{C}{\kappa
\beta^2\delta}P_{+}(\Omega,X)+\frac{C}{\kappa\beta\delta}P_{+}(\Omega,X)\\
&\quad \le (1-\delta)^{-1}P_{\kappa}(\Omega,X)
  +\frac{C}{\kappa\beta^2\delta}P_{+}(\Omega,X),
\end{align*}
where we used~\eqref{eq:perimeters of Omega i and Omega} in the last step
(and absorbed a factor $2$ into the constant $C$).
By~\eqref{eq:measure of balls from boxing inequality} we have
\begin{align*}
\Vert \ch_{K_i}- \ch_{\Omega}\Vert_{L^1(X)}
&\le \sum_{j\in\N}\mu(B(x_j^i,\widetilde{r}_j^i))+
\sum_{k\in\N} \mu(B(y_k^i,\widetilde{s}_k^i))\\
&\le C\sum_{j\in\N}\mu(B(x_j^i,r_j^i/5))+
C_d\sum_{k\in\N} \mu(B(y_k^i,s_k^i))\\
&\le \frac{C}{\beta}\sum_{j\in\N}\mu(B(x_j^i,r_j^i/5)\cap \overline{\Omega_i}\setminus\Omega)+
\frac{C_d}{i}\sum_{k\in\N} \frac{\mu(B(y_k^i,s_k^i))}{s_k^i}\\
&\le \frac{C\mu(\overline{\Omega_i}\setminus\Omega)}{\beta}+
\frac{C_d}{i}\sum_{k\in\N} \frac{\mu(B(y_k^i,s_k^i))}{s_k^i},
\end{align*}
since the balls $B(x_j^i,r_j^i/5)$ are disjoint.
Now by the fact that $\ch_{\Om_i}\to \ch_{\Om}$ in $L^1(X)$ and
\eqref{eq:second estimate for balls},
we obtain $\ch_{K_i}\to \ch_{\Omega}$ in $L^1(X)$. Thus by Lemma~\ref{lem:lower semicontinuity for P+}
\begin{align*}
P_+(\Omega,X)
&\le \liminf_{i\to\infty}P(K_i,X)\\
&\le (1-\delta)^{-1}P_{\kappa}(\Omega,X)+ \frac{C}{\kappa\beta^2\delta}P_{+}(\Omega,X).
\end{align*}
Letting $\kappa\to\infty$ and then $\delta\to 0$, we obtain the result.
\end{proof}

\begin{corollary}\label{cor:comparison of P plus and P kappa}
Under the assumptions of Proposition~\ref{prop:comparison of P plus and P kappa},
for any Borel set $D\subset X$ we have
\[
P_+(\Omega,D)=\lim_{\kappa\to\infty}P_{\kappa}(\Omega,D).
\]
\end{corollary}
\begin{proof}
For any Borel set $A\subset X$ and any $\kappa>0$, we have 
$P_+(\Omega,A)\ge P_{\kappa}(\Omega,A)$. Thus by the measure property of
$P_+(\Om,\cdot)$ proved in the Appendix, and by
Proposition~\ref{prop:comparison of P plus and P kappa},
\begin{align*}
P_+(\Omega,D)
&=P_+(\Omega,X)-P_+(\Omega,X\setminus D)\\
&=\lim_{\kappa\to\infty}P_{\kappa}(\Omega,X)-P_+(\Omega,X\setminus D)\\
&\le \limsup_{\kappa\to\infty}\big(P_{\kappa}(\Omega,X)-P_{\kappa}(\Omega,X\setminus D)\big)\\
&= \lim_{\kappa\to\infty}P_{\kappa}(\Omega,D).
\end{align*}
\end{proof}

\begin{proof}[{End of the proof of Theorem~\ref{thm:dirichletT}}]
Note that by~\eqref{eq:exterior measure density condition in theorem consequence},
$\Omega$ satisfies the assumptions of Proposition
\ref{prop:comparison of P plus and P kappa} (and thus Corollary~\ref{cor:comparison of P plus and P kappa}) with $\beta=\kappa$.
By Cavalieri's principle,
Corollary~\ref{cor:comparison of P plus and P kappa},
and Lebesgue's monotone convergence theorem,
we have
\begin{align*}
&\int_{\partial^*\Omega}|T_+u-f|\, dP_{\kappa}(\Omega,\cdot)\\
&\qquad\qquad= \int_0^{\infty} P_{\kappa}(\Omega,\partial^*\Omega\cap \{|T_+u-f|>t\})\,dt\\
&\qquad\qquad\to \int_0^{\infty} P_{+}(\Omega,\partial^*\Omega\cap \{|T_+u-f|>t\})\,dt\qquad\textrm{as }\kappa\to\infty\\
&\qquad\qquad= \int_{\partial^*\Omega}|T_+u-f|\, dP_{+}(\Omega,\cdot).
\end{align*}
This proves~\eqref{eq:p-1-kappa-infty}, thus completing the proof
of Theorem~\ref{thm:dirichletT}.

\end{proof}

\begin{corollary}
Let $\Omega\subset X$ satisfy the assumptions of Theorem~\ref{thm:dirichletT}, and
let $f\in \Lip(X)$ be boundedly supported. Then the minimization problem $(T)$ of Definition~\ref{def:dirichlet} has a solution.
\end{corollary}
\begin{proof}
For every $p>1$, there exists a \p-harmonic function $u_p$  in $\Om$ such that $u_p|_{X\setminus\Om}=f$. Then the result follows by combining Lemma~\ref{lem:bv embedding lemma} and Theorem~\ref{thm:dirichletT}.
\end{proof}

\section{Dirichlet Problems~(T), (B) and perturbation of the domain}

From the definition of $P_+(\Om,\cdot)$ it is clear that Problem~(T)
of Definition~\ref{def:dirichlet} is associated with approximating 
the bounded open set $\Om$ from inside. Moreover, if for each $k\in\N$ we have
$\Om_k\Subset\Om$
such that $\Om=\bigcup_{k\in\N}\Om_k$, and $v_{p_k}$ is the $p_k$-harmonic solution to the Dirichlet
problem on $\Om_k$ with boundary data $f$, then under reasonable hypotheses on $\Om$ we 
have $v_{p_k}\to u$ with $u$ a solution to Problem~(T) in $\Om$ with boundary data $f$. Indeed, suppose that
for each $p>1$ there are constants $C_p\ge 1$ and $\beta_p>0$ such that 
$\Om$ satisfies the condition that whenever $u_p$ is a $p$-harmonic solution to the Dirichlet problem on $\Om$
with boundary data $f$ we have 
\begin{equation}\label{eq:fat}
\osc_{\overline{\Om}\cap B(x,\rho)}u_p\le \osc_{\partial\Om\cap B(x,r)}f+C_p\left(\frac{\rho}{r}\right)^{\beta_p}
\end{equation}
for all $x\in\partial\Om$ and $0<\rho<r/2$. Let  $p_k$ be any sequence as obtained in 
Lemma~\ref{lem:bv embedding lemma}. Take a sequence $\eps_k\to 0$ as $k\to\infty$, 
and let $L>0$ such that $f$ is $L$-Lipschitz
continuous. For each $k\in\N$ we can fix $0<r_k<\diam(\Om)/2$ such that $4Lr_k<\eps_k$. Then
by~\eqref{eq:fat}, whenever $x\in\partial\Om$ and $0<\rho<r_k/2$ and $y\in B(x,\rho)\cap\Om$,
\begin{align*}
|u_{p_k}(y)-f(y)|&\le |u_{p_k}(y)-u_{p_k}(x)|+|f(x)-f(y)|\\
  &\le 2Lr_k+C_{p_k}\left(\frac{\rho}{r_k}\right)^{\beta_{p_k}}+Lr_k.
\end{align*}
We can then choose $0<\rho_k<r_k/2$ such that 
\[
C_{p_k}\left(\frac{\rho_k}{r_k}\right)^{\beta_{p_k}}<Lr_k,
\]
in which case for $y\in B(x,\rho_k)\cap\Om$ we have
\[
|u_{p_k}(y)-f(y)|\le 4Lr_k<\eps_k.
\]
Now if we choose $\Om_k\Subset\Om$ such that $0<\dist(\partial\Om_k,\partial\Om)<\rho_k$, we have
by the comparison principle for $p$-harmonic functions (see~\cite{BB}) that 
$|v_k-u_{p_k}|<\eps_k$ on $\Om_k$. It then follows that $v_k\to u$ 
in $L^1(X)$ as $k\to\infty$, where we know from
the previous section that $u$ satisfies the Dirichlet problem~(T) on $\Om$ with boundary data $f$.
Examples of domains where~\eqref{eq:fat} hold include the domains whose complements are uniformly
$1$-fat, see~\cite{BjMcSh}; in particular, domains whose complements are porous satisfy this requirement.

In contrast to problem~(T), the Dirichlet problem~(B) of Definition~\ref{def:dirichlet} is associated with approximation of
$\Om$ from \emph{outside}, as we will see next.

Let $\Om\subset X$ be a nonempty bounded open set with
$X\setminus\overline{\Om}\neq \emptyset$.
Let $\Om_k\subset X$, $k\in\N$, be a sequence of bounded open sets such that 
$\overline{\Om}=\bigcap_{k\in\N}\Om_k$ and
$\Om_{k+1}\Subset\Om_k$ for each $k\in\N$.
Note that since $X\setminus\overline{\Om}\neq \emptyset$ is open,
we have $\mu(X\setminus\overline{\Om})>0$.
Thus also $\capa_1(X\setminus \Om_k)\ge \mu(X\setminus \Om_k)>0$ for all
sufficiently large $k$, and so we may as well assume that this is true for all $k\in\N$.
For each $k\in\N$, fix a
decreasing sequence $(p_{k,m})_m$ such that $p_{k,m}>1$ and 
\[\lim_{m\rightarrow\infty}p_{k,m}=1.
\]
Let $f\in\Lip(X)$ be boundedly supported, and
let $u_{k,m}\in N^{1,p_{k,m}}(X)$ be the $p_{k,m}$-harmonic function
solving the Dirichlet problem on $\Om_k$ with boundary data $f$.
According to Lemma~\ref{lem:bv embedding lemma}, by
 passing to a subsequence of
$(p_{k,m})_m$ (not relabeled), we find a function $u_k\in \BV(X)$ with $u_k=f$
on $X\setminus \Omega_k$ such that
$u_{k,m}\to u_k$ in $L^1(X)$ as $m\to\infty$. We have
\begin{align*}
\int_{\Omega_k}g_{u_{k,m}}\,d\mu
&\le \left(\int_{\Omega_k}g^{p_{k,m}}_{u_{k,m}}\,d\mu\right)^{1/p_{k,m}}\mu(\Omega_k)^{1-1/p_{k,m}}\\
&\le \left(\int_{\Omega_k}g^{p_{k,m}}_fd\mu\right)^{1/p_{k,m}}\mu(\Omega_k)^{1-1/p_{k,m}}.
\end{align*}
As in Section~\ref{sec:limit is least gradient function}, $g_{u_{k,m}}$  always
denotes the minimal $p$-weak upper gradient of $u_{k,m}$, and
for a Lipschitz function $f$, $g_f$ denotes the minimal $p$-weak upper gradient of
$f$ for any $p>1$.
By~\eqref{eq:comparing variation and upper gradients},
\[
\Vert Du_{k,m}\Vert(X)\le \int_X g_{u_{k,m}}\,d\mu
=\int_{\Om_k} g_{u_{k,m}}\,d\mu+\int_{X\setminus\Om_k} g_{f}\,d\mu.
\]
By the lower semicontinuity of the total variation,
\[
\Vert Du_k\Vert(X)\le \liminf_{m\rightarrow\infty}\Vert Du_{k,m}\Vert(X)\le 
\int_X g_f\,d\mu.
\]
We have $|f|\le M$ on $X$ for some $M\ge 0$.
By the comparison principle, we also have $|u_{k,m}|\le M$, and then $|u_k|\le M$.
Thus for all $k\in\N$,
\[
\Vert u_k\Vert_{L^1(X)}\le \Vert u_k\Vert_{L^{\infty}(X)}\mu(\Omega_k)
+\Vert f\Vert_{L^1(X\setminus\Om_k)}
\le M\mu(\Omega_1)+\Vert f\Vert_{L^1(X)}<\infty.
\]
Then by the compact embedding given in~\cite[Theorem~3.7]{Miranda03},
by passing to a subsequence of $k$ (not relabeled),
we obtain $u_k\to u$ in $L^1(X)$ as $k\to\infty$, for $u\in \BV(X)$.
By passing to a further subsequence (not relabeled), we can assume  that 
\begin{equation}\label{eq:L1 convergence of uks to u}
\int_X|u_k-u|\, d\mu<1/k.
\end{equation}

\begin{theorem}\label{thm:ProblemB}
Let $\Om\subset X$ be a nonempty bounded open set such that
$X\setminus\overline{\Om}\neq \emptyset$,
and suppose that $f\in\Lip(X)$ is boundedly supported.
Let $u$ be the function constructed above. 
Then $u$ solves Problem~(B) of Definition~\ref{def:dirichlet}.
\end{theorem}

Note that in this section we do not need $\Om$ to satisfy the extra conditions
imposed in Section~\ref{sec:Dirichlet problem trace}.

\begin{proof}
Take a test function $\psi\in \BV(X)$ such that $\psi=0$ on $X\setminus\overline{\Om}$. 
We can choose a sequence $(\Psi_j)\subset \liploc(X)$
such that $\Psi_j\to u+\psi$ in $L^1(X)$ as
$j\to\infty$ and 
\[
\lim_{j\to\infty}\int_X g_{\Psi_j}\,d\mu= \Vert D(u+\psi) \Vert(X),
\]
where $g_{\Psi_j}$ is the minimal $p$-weak upper gradient of $\Psi_j$ in
$X$, for any $p>1$, see the discussion on page \pageref{p-weak upper gradients in total variation}.
Then also $g_{\Psi_j}\,d\mu\to d\Vert D(u+\psi) \Vert$ weakly* in the sense of measures
(see e.g.~\cite[Proposition 3.8]{HKLL}),
so that for each $k\in\N$,
\begin{equation}\label{eq:upper semicontinuity for Psijs}
\begin{split}
\limsup_{j\to\infty}\int_{\Om_k} g_{\Psi_j}\,d\mu
&\le \Vert D(u+\psi)\Vert(\overline{\Om_k})\\
&=\Vert D(u+\psi)\Vert(\overline{\Om})+\Vert D(u+\psi)\Vert(\overline{\Om_k}\setminus\overline{\Om})\\
&=\Vert D(u+\psi)\Vert(\overline{\Om})+\Vert Df\Vert(\overline{\Om_k}\setminus\overline{\Om})
\end{split}
\end{equation}
since $u+\psi=f$ in $\overline{\Om_k}\setminus\overline{\Om}$.
For each $k\in\N$, let $\eps_k:=\dist(\overline{\Omega},X\setminus\Omega_k)>0$,
and let $\eta_{\eps_k}\in\Lip_c(X)$ with $0\le \eta_{\eps_k}\le 1$,
$\eta_{\eps_k}=1$ in $\overline{\Omega}$, $\eta_{\eps_k}=0$ in $X\setminus\Omega_k$,
and $g_{\eta_{\eps_k}}\le 1/\eps_k$. We set
\[
\psi_{j,k}:=\eta_{\eps_k}\Psi_j+(1-\eta_{\eps_k})f.
\]
By the Leibniz rule of~\cite[Lemma~2.18]{BB},
\begin{equation}\label{eq:Leibniz rule estimate for psijk}
\begin{split}
g_{\psi_{j,k}} &\le 
g_{\Psi_j}\eta_{\eps_k} + g_{f}(1-\eta_{\eps_k}) + 
g_{\eta_{\eps_k}}|\Psi_j-f| \\
&\le g_{\Psi_j}\ch_{\Om_k} + g_{f}\ch_{\Om_k\setminus\overline{\Om}} +
 (1/\eps_k)|\Psi_j-f|\ch_{\Om_k\setminus\overline{\Om}}.
\end{split}
\end{equation}
Note that this function agrees with $f$ in
$X\setminus\Omega_k$.

As noted above, $|u_k|\le M$ for all $k\in\N$, and so $|u|\le M$. As truncation decreases total variation, we 
can also assume that 
$|u+\psi|\le M$ and that the approximating functions $\Psi_j$ also satisfy $|\Psi_j|\le M$.
Then we have (assuming $p_{k,m}<2$ and $M\ge 1$)
\begin{align*}
\int_{X\setminus\overline{\Om}}|\Psi_j-f|^{p_{k,m}}\, d\mu
&\le 2^{p_{k,m}-1}M^{p_{k,m}-1}\int_{X\setminus\overline{\Om}}|\Psi_j-f|\, d\mu\\
& \le 2 M\int_{X\setminus\overline{\Om}}|\Psi_j-f|\, d\mu.
\end{align*}
For each $k\in\N$, by~\eqref{eq:upper semicontinuity for Psijs} we can choose $j_k\in\N$ large enough so that
\begin{equation}\label{eq:choice of jk}
\int_{\Om_k}g_{\Psi_{j_k}}\, d\mu
\le \Vert D(u+\psi)\Vert(\overline{\Om})+\Vert Df\Vert(\overline{\Om_k}\setminus\overline{\Om})+1/k
\end{equation}
and
\[
\int_{X\setminus\overline{\Om}}|\Psi_{j_k}-f|\, d\mu\le \frac{\eps_k^2}{2M},
\]
and then it follows that for all $m$
\begin{equation}\label{eq:L1 closeness of Psij and f}
\int_{X\setminus\overline{\Om}}|\Psi_{j_k}-f|^{p_{k,m}}\, d\mu\le \eps_k^2.
\end{equation}
Then we choose $m_k\in\N$ large enough so that $1<p_{k,m_k}<1+1/k$,
\begin{equation}\label{eq:L1 closeness of ukmk and uk}
\int_X|u_{k,m_k}-u_k|\, d\mu<1/k,
\end{equation}
and
\[
\left(\int_{\Om_k}g_{\Psi_{j_k}}^{p_{k,m_k}}\, d\mu\right)^{1/p_{k,m_k}}
\le \int_{\Om_k}g_{\Psi_{j_k}}\, d\mu+1/k.
\]
It then follows by~\eqref{eq:choice of jk} that 
\begin{equation}\label{eq:estimate for upper gradients of psis by variation}
\left(\int_{\Om_k}g_{\Psi_{j_k}}^{p_{k,m_k}}\, d\mu\right)^{1/p_{k,m_k}}
\le \Vert D(u+\psi)\Vert(\overline{\Om})+\Vert Df\Vert(\overline{\Om_k}\setminus\overline{\Om})+2/k.
\end{equation}
Combining~\eqref{eq:L1 closeness of ukmk and uk} with~\eqref{eq:L1 convergence of uks to u}, we have $u_{k,m_k}\to u$ in $L^1(X)$ as $k\to\infty$.
By lower semicontinuity of the total variation, H\"older's inequality,
and the fact that $\psi_{j_k,k}$ can be used to test the $p_{k,m}$-harmonicity of
$u_{k,m_k}$, we get 
\begin{align*}
&\Vert Du\Vert(\overline{\Om})\le \liminf_{k\to\infty}\int_{\Omega_k} g_{u_{k,m_k}}\, d\mu\\
&\ \ \ \le \liminf_{k\to\infty}\left(\int_{\Om_k} g_{u_{k,m_k}}^{p_{k,m_k}}\, d\mu\right)^{1/p_{k,m_k}}\\
&\ \ \ \le \liminf_{k\to\infty}\left(\int_{\Om_k}g_{\psi_{j_k,k}}^{p_{k,m_k}}\, d\mu\right)^{1/p_{k,m_k}}\\
&\ \ \ \le  \limsup_{k\to\infty} \, \left(\int_{\Om_k}g_{\Psi_{j_k}}^{p_{k,m_k}}\, d\mu\right)^{1/p_{k,m_k}}
+\limsup_{k\to\infty}\left(\int_{\Omega_k\setminus\overline{\Om}}g_{f}^{p_{k,m_k}}\, d\mu\right)^{1/p_{k,m_k}}\\
&\ \ \qquad +\limsup_{k\to\infty} \frac{1}{\eps_k}
\left(\int_{\Om_k\setminus\overline{\Om}}|\Psi_{j_k}-f|^{p_{k,m_k}}\, d\mu\right)^{1/p_{k,m_k}}\\
&\ \ \ \le \Vert D(u+\psi)\Vert(\overline{\Om})
+0+\liminf_{k\to\infty} \eps_k^{2/p_{k,m_k}-1}\\
&\ \ \ =\Vert D(u+\psi)\Vert(\overline{\Om}).
\end{align*}
In the above, we used~\eqref{eq:Leibniz rule estimate for psijk} to arrive at the fourth inequality, and
\eqref{eq:estimate for upper gradients of psis by variation},~\eqref{eq:L1 closeness of Psij and f} in 
obtaining the penultimate inequality.
\end{proof}

\section{Alternate definitions of functions of least gradient}

In this section we consider possible definitions of what it means for a function
$u\in \BV(\Omega)$ to be of least gradient in an open set $\Omega\subset X$.
This is not to be conflated with the notions of solutions to the Dirichlet problems 
studied in the previous sections, as such solutions must in addition satisfy a boundary condition.

Recall that by $\BV_c(\Om)$ we mean the collection of functions $\psi\in \BV(\Om)$
such that $\supp(\psi)\subset\Om$, and by $\BV_0(\Om)$ we mean the collection of functions $\psi\in \BV(\Om)$
for which $T_+\psi$ exists and $T_+\psi=0$ $\mathcal{H}$-a.e.~in $\partial \Om$.

In addition to the class $\BV_0(\Om)$, in this section we also consider a larger class of test functions,
the class $wk-\BV_0(\Om)$ of functions $\psi\in \BV(\Om)$ such that
\[
\lim_{r\to0^+}\frac{1}{\mu(B(x,r))}\int_{B(x,r)\cap\Om}|\psi|\, d\mu=0
\]
for $\mathcal{H}$-a.e.~$x\in\partial\Om$.

Note that $\BV_0(\Om)\subset wk-\BV_0(\Om)$.
If $P(\Om,X)<\infty$, 
recall that $\mathcal H(\partial^*\Om\setminus \Sigma_\gamma\Om)=0$, where
$\Sigma_\gamma\Om$ was defined after~\eqref{eq:density of E}.
Note that if $f\in wk-\BV_0(\Om)$, then for $\mathcal{H}$-a.e.~$x\in (I_\Om\cap\partial\Om)\cup\Sigma_\gamma\Om$ we have
\[
\lim_{r\to0^+}\frac{1}{\mu(B(x,r)\cap\Om)}\int_{B(x,r)\cap\Om}|u|\, d\mu=0,
\]
that is, $T_+u(x)=0$. The philosophy here is that the trace of 
test functions in $O_\Om$, from the point of view of solving a Dirichlet problem for BV functions, should not affect
the solution to the problem.

Recall that $\Vert Du\Vert^s$ denotes the singular part of the variation measure with 
respect to $\mu$.

\begin{definition}
We have the following alternative notions of functions of least gradient. Let $u\in\BV(\Om)$.
\begin{enumerate}
\item We say that $u$ is of \emph{least gradient} in $\Omega$ if whenever 
$\psi\in \BV_c(\Omega)$, we have
$\Vert Du\Vert(\Om)\le \Vert D(u+\psi)\Vert(\Om)$.
\item We say that $u$ is \emph{globally of least gradient} in $\Omega$ if whenever $\psi\in \BV_0(\Omega)$,
we have $\Vert Du\Vert(\Omega)\le \Vert D(u+\psi)\Vert(\Omega)$.
\item We say that $u$ is \emph{globally of least gradient in the wider sense} in $\Omega$ if whenever 
$\psi\in wk-\BV_0(\Omega)$, we have $\Vert Du\Vert(\Omega)\le \Vert D(u+\psi)\Vert(\Omega)$.
\item We say that $u$ is of \emph{least gradient in the sense of Anzellotti} in $\Omega$
if whenever
$\psi\in  wk-\BV_0(\Omega)$ with $\Vert D\psi\Vert^s\ll\Vert Du\Vert^s$, we have
$\Vert Du\Vert(\Omega)\le \Vert D(u+\psi)\Vert(\Omega)$.
\item We say that $u$ is of \emph{least gradient relative to Lipschitz functions} in $\Om$
if whenever
$\varphi\in \liploc(\Omega)$ with $\varphi-u\in wk-\BV_0(\Om)$, we have
$\Vert Du\Vert(\Omega)\le \Vert D\varphi\Vert(\Omega)$.
\end{enumerate}
\end{definition}

\begin{lemma}\label{lem:least gradient wrt lip functions implies least gradient}
Let $\Omega\subset X$ be an open set with $\mathcal H(\partial\Om)<\infty$. 
Let $u$ be of least gradient relative to Lipschitz functions in $\Omega$. Then $u$ is
globally of least gradient in the wider sense in $\Omega$.
\end{lemma}

\begin{proof}
Let $\psi\in wk-\BV_0(\Omega)$.
By~\cite[Corollary~6.8]{LahSh}, we find a sequence
$(\varphi_{j})\subset \liploc(\Om)$ such that
$\varphi_{j}-(u+\psi)\in wk-\BV_0(\Om)$,
$\varphi_{j}\to u+\psi$ in $L^1(\Om)$, and
$\Vert D\varphi_{j}\Vert(\Om)\to\Vert D(u+\psi)\Vert(\Om)$ as $j\to\infty$. Then also
$\varphi_{j}-u\in wk-\BV_0(\Om)$. 
Thus by the fact that $u$ is of least gradient relative to Lipschitz functions,
\[
\Vert Du\Vert(\Om)\le \lim_{j\to\infty}\Vert D\varphi_{j}\Vert(\Om)=
\Vert D (u+\psi)\Vert(\Om).
\]
\end{proof}

\begin{proposition}
Let $\Omega\subset X$ be an open set with $\mathcal H(\partial\Om)<\infty$, and 
let $u\in\BV(\Omega)$. Then the alternative definitions (1)-(5) of $u$ being of least
gradient in $\Om$ are equivalent.
\end{proposition}

\begin{proof}
\hfill
\begin{itemize}
\item $(1)\implies (3)$: This follows from the fact that for any $\psi\in wk-\BV_0(\Omega)$
we find a sequence of functions $\psi_k\in \BV_c(\Omega)$ such that
$\psi_k\to\psi$ in $\BV(\Omega)$, by~\cite[Theorem~6.10]{LahSh}.
\item $(3)\implies (2)\implies (1)$: These implications are trivial.
\item $(3)\implies (4)$: This is trivial.
\item $(4)\implies (5)$: Let $\varphi\in \liploc(\Omega)$ with $\varphi-u\in wk-\BV_0(\Om)$.
Then clearly $\Vert D(\varphi-u)\Vert^s= \Vert Du\Vert^s$, so we have
\[
\Vert Du\Vert(\Omega)\le \Vert D(u+(\varphi-u))\Vert(\Omega)=\Vert D\varphi\Vert(\Omega).
\]
\item $(5)\implies (3)$: This is Lemma~\ref{lem:least gradient wrt lip functions implies least gradient}.
\end{itemize}
\end{proof}

Consider again the Dirichlet problem (T) of Definition~\ref{def:dirichlet}.
We say that $u\in \BV(\Omega)$
solves the Dirichlet problem (T)
relative to Lipschitz functions if
\[
\Vert Du\Vert(\Omega)+\int_{\partial^*\Omega}|T_+u-f|\, dP_+(\Omega,x)\le \Vert Dv\Vert(\Omega)
\]
for all $v\in\Lip(X)$ with $v=f$ in $X\setminus\Omega$. Note that the boundary term vanishes
for $v$.

\begin{proposition}
Let $\Omega\subset X$ satisfy the assumptions of Lemma~\ref{lem:use1},
and let $f\in\Lip(X)$ be boundedly supported.
If $u\in \BV(\Omega)$
solves the Dirichlet problem (T)
relative to Lipschitz functions, then $u$ solves the Dirichlet problem (T).
\end{proposition}
\begin{proof}
Pick a sequence of functions $\psi_k\in\Lip(X)$ given by Lemma~\ref{lem:use1}.
Then
\begin{align*}
&\Vert Du\Vert(\Omega)+ \int_{\partial^*\Omega}|T_+u-f|\, dP_+(\Omega,\cdot)
\le \lim_{k\rightarrow\infty}\Vert D\psi_k\Vert(\Omega)\\
&\qquad\qquad\qquad =\Vert Dv\Vert(\Omega)+\int_{\partial^*\Omega}|T_+v-f|\, dP_+(\Omega,\cdot).
\end{align*}
\end{proof}

\section*{Appendix: Proof that $P_+(\Om,\cdot)$ is a Radon measure}
\renewcommand{\thesection}{\Alph{section}}
\setcounter{section}{1}
\setcounter{equation}{0}

Let $\Omega\subset X$ be an open set with $P_+(\Omega,X)<\infty$.
In showing that $P_+(\Omega,\cdot)$ is a Radon measure on $X$, we rely on the following theorem,
due to De Giorgi and Letta~\cite[Theorem~5.1(3,5)]{DeLett},
whose proof can also be found in e.g.~\cite[Theorem 1.53]{AFP}. 

\begin{theorem}\label{thm:DeGiorgiLetta}
Let $\nu$ be a function defined on the open sets of $X$ taking values in $[0,\infty]$
such that $\nu(\emptyset)=0$, $\nu(U_1)\le \nu(U_2)$ for any open sets $U_1\subset U_2$,
and such that the following properties hold:
\begin{enumerate}
\item If $U_1,U_2\subset X$ are open sets, then $\nu(U_1\cup U_2)\le \nu(U_1)+\nu(U_2)$.
\item If $U_1,U_2\subset X$ are disjoint open sets, then
\[
\nu(U_1\cup U_2)\ge \nu(U_1)+\nu(U_2).
\]
\item For any open set $U\subset X$, $\nu(U)=\sup\{\nu(V),\,V\textrm{ open, } V\Subset U\}$.
\end{enumerate}
Then the extension of $\nu$ to all sets in $X$, defined by
\[
\nu(A):=\inf\{\nu(U):\, U\textrm{ open, }U\supset A\},\quad A\subset X,
\]
is a Borel outer measure.
\end{theorem}

Clearly we have $P_+(\Omega,\emptyset)=0$ and  $P_+(\Omega,U_1)\le P_+(\Omega,U_2)$ for any open sets $U_1\subset U_2$.
We also note that if $U_1,U_2\subset X$ are two disjoint open sets, then
\begin{equation}\label{eq:additivity in disjoint open sets}
P_+(\Omega,U_1\cup U_2)=P_+(\Omega,U_1)+P_+(\Omega,U_2),
\end{equation}
verifying property (2).
Next we prove two ``pasting" lemmas, which are analogs of~\cite[Lemma~3.3]{Miranda03}.
In this section, $g_u$ always denotes the minimal $1$-weak upper gradient of $u$.

\begin{lemma}\label{lem:paste}
Let $U_1,U_2\subset X$ be open sets such that $U_1$ is bounded and
$\partial U_1\cap\partial U_2=\emptyset$. Then there exists a function
$\eta\in\Lip_c(X)$ with $0\le \eta\le 1$ such that whenever $u\in\liploc(U_1)$
and $v\in \liploc(U_2)$, the function $w:=\eta u+(1-\eta)v\in\liploc(U_1\cup U_2)$ satisfies
\[
\int_{U_1\cup U_2} g_w\,d\mu\le \int_{U_1} g_u\,d\mu+\int_{U_2} g_v\,d\mu+C(U_1,U_2)\int_{U_1\cap U_2}|u-v|\, d\mu,
\]
and whenever $h\in L_{\loc}^1(U_1\cup U_2)$,
\[
\int_{U_1\cup U_2} |w-h|\, d\mu\le \int_{U_1}|u-h|\, d\mu+\int_{U_2}|v-h|\, d\mu,
\]
where $C(U_1,U_2)$ depends solely on $\dist(\partial U_1,\partial U_2)$ and is independent of $u,v$.
\end{lemma}

\begin{proof}
Let $\eta$ be a Lipschitz map from $U_1\cup U_2$ to $[0,1]$ such that $\eta=1$ on $U_1\setminus U_2$ and $\eta=0$ on
$U_2\setminus U_1$. Then the desired results follow from 
the Leibniz rule~\cite[Lemma 2.18]{BB}.
\end{proof}

\begin{lemma}\label{lem:paste2}
Let $U_1'\Subset U_1$ and $U_2'\Subset U_2$ be open sets. Then there exists a function $\eta\in\Lip_c(U_1)$ with $0\le \eta\le 1$ such that whenever $u\in\liploc(U_1)$ and $v\in \liploc(U_2)$, the function $w:=\eta u+(1-\eta)v\in\Lip(U_1'\cup U_2')$ satisfies
\[
\int_{U_1'\cup U_2'}g_w\,d\mu\le \int_{U_1} g_u\,d\mu+\int_{U_2}g_v\,d\mu+C(U_1',U_1)
\int_{U_1\cap U_2} |u-v|\,d\mu.
\]
\end{lemma}

\begin{proof}
Let $\eta\in\Lip_c(X)$ such that $0\le\eta\le 1$, $\eta=1$ in $U_1'$,
and $\eta=0$ in $X\setminus U_1$. Then the desired result again follows from the Leibniz rule.
\end{proof}

For an open set $W\subset X$ and $\delta>0$, let
\[
W_\delta:=\{x\in W :\, \dist(x,X\setminus W)>\delta\}.
\]

\begin{lemma}\label{lem:pasting with annulus}
Let $W\subset X$ be open and $0<\delta_1<\delta_2$. Then
\[
P_+(\Om, W)\le P_+(\Om,W_{\delta_1})+P_+(\Om, W\setminus\overline{W}_{\delta_2}).
\]
\end{lemma}

\begin{proof}
We set $U_1=W_{\delta_1}$ and $U_2=W\setminus \overline{W}_{\delta_2}$.
We can find sequences $(u_i)\subset \liploc(U_1)$ and $(v_i)\subset \liploc(U_2)$ such that
$u_i$ vanishes in $U_1\setminus\Om$, $v_i$ vanishes in $U_2\setminus\Om$,
$u_i-\ch_\Om\to 0$ in $L^1(U_1)$, $v_i-\ch_\Om\to 0$ in $L^1(U_2)$, and
\[
\lim_{i\to\infty}\int_{U_1}g_{u_i}\,d\mu=P_+(\Om,U_1),\qquad \
\lim_{i\to\infty}\int_{U_2}g_{v_i}\,d\mu= P_+(\Om,U_2).
\]
Applying Lemma~\ref{lem:paste} with $u=u_i$ and $v=v_i$, we obtain functions
$w_i\in \liploc(U_1\cup U_2)=\liploc(W)$
such that $w_i=u_i$ in $U_1\setminus U_2$, $w_i=v_i$ in $U_2\setminus U_1$, 
\[
   \int_W|w_i-\ch_\Om|\, d\mu\le \int_{U_1}|u_i-\ch_\Om|\, d\mu+\int_{U_2}|v_i-\ch_\Om|\, d\mu,
\]
and
\[
\int_W g_{w_i}\,d\mu\le \int_{U_1} g_{u_i}\,d\mu+\int_{U_2} g_{v_i}\,d\mu
+C(U_1,U_2)\int_{U_1\cap U_2}|u_i-v_i|\, d\mu.
\]
Furthermore, by the construction of $w_i$, we see that $w_i$ vanishes in
$W\setminus\Om$.
From the first of the above two inequalities we see that $w_i-\ch_\Om\to 0$ in $L^1(W)$, and hence
by the second of the above two inequalities,
\begin{align*}
P_+(\Om,W)
&\le \liminf_{i\to\infty}\int_W g_{w_i}\,d\mu\\
&\le P_+(\Om,U_1)+P_+(\Om, U_2)
  +\limsup_{i\to\infty} \int_{U_1\cap U_2}|u_i-v_i|\, d\mu.
\end{align*}
Note that 
\[
 \int_{U_1\cap U_2}|u_i-v_i|\, d\mu\le \int_{U_1\cap U_2}|u_i-\ch_\Om|+|v_i-\ch_\Om|\, d\mu\to 0
\]
as $i\to\infty$.
The desired conclusion now follows.
\end{proof}

\begin{lemma}\label{lem:thin annulus has small measure}
Let $W\subset X$ be an open set and let $\eps>0$. Then for some $\delta>0$ we have
\[
P_+(\Om, W\setminus \overline{W}_{\delta})<\eps.
\]
\end{lemma}

\begin{proof}
Let $(\delta_k)$ be a strictly decreasing sequence of positive numbers such that 
$\lim_{k\to\infty}\delta_k=0$. For integers $k\ge 2$, let $V_k:=W_{\delta_{2k}}\setminus\overline{W}_{\delta_{2k-3}}$. 
Note then that $\{V_{2k}\}_{k\in\N}$ is a collection of pairwise disjoint open sets, 
$\{V_{2k+1}\}_{k\in\N}$ is a collection of pairwise disjoint open sets, and $V_k\cap V_j$ is non-empty if and only if 
$|k-j|\le 1$. From
\eqref{eq:additivity in disjoint open sets} we see that 
\[
\sum_{k=1}^\infty P_+(\Om, V_{2k}) \le P_+\left(\Om,\bigcup_{k=1}^{\infty}V_{2k}\right)\le P_+(\Om,W)<\infty,
\]
and so we can find $k_1\in\N$ such that 
\begin{equation*}
\sum_{k=k_1}^\infty P_+(\Om, V_{2k})<\frac{\eps}{2}.
\end{equation*}
Analogously, we can find $k_2\in\N$ such that 
\begin{equation*}
\sum_{k=k_2}^\infty P_+(\Om, V_{2k+1})<\frac{\eps}{2}.
\end{equation*}
Thus by choosing $k_{\eps}=2\max\{k_1,k_2\}+1$, we obtain
\begin{equation}\label{eq:tail-end}
\sum_{k=k_\eps}^\infty P_+(\Om,V_k)<\eps.
\end{equation}
For each $k\ge k_{\eps}$ we can choose a
sequence $(u_{k,i})\subset \liploc(V_k)$ such that $u_{k,i}$ vanishes in 
$V_k\setminus\Om$, 
\[
 \int_{V_k}|u_{k,i}-\ch_\Om|\, d\mu\le 2^{-i-k}\, \min\left\{1,\frac{1}{C\left(\bigcup_{j=k_\eps}^k V_j, V_{k+1}\right)}\right\}, 
 \]
 and
\[
\int_{V_k} g_{u_{k,i}}\,d\mu\le P_+(\Om, V_k)+2^{-i-k}.
\]
Fix $i\in\N$. We construct a function $w_i$ inductively as follows. For $k=k_\eps$ we set $w_{i,k}=u_{k,i}$. We apply
Lemma~\ref{lem:paste} with $U_1=V_{k_\eps}$, $U_2=V_{k_\eps+1}$, $u=w_{i,k_\eps}$, and
$v=u_{k+1,i}$ to obtain $w_{i,k_\eps+1}\in\liploc(V_{k_\eps}\cup V_{k_\eps+1})$.
Note that $w_{i,k_\eps+1}$ vanishes in $(V_{k_\eps}\cup V_{k_\eps+1})\setminus\Om$, 
\[
  \int_{V_{k_\eps}\cup V_{k_\eps+1}}|w_{i,k_\eps+1}-\ch_\Om|\, d\mu\le 2^{-i}(2^{-k_\eps}+2^{-k_\eps-1}),
\]
and 
\[
\int_{V_{k_\eps}\cup V_{k_\eps+1}} g_{w_{i,k_\eps+1}}\,d\mu\le P_+(\Om,V_{k_\eps})+P_+(\Om, V_{k_\eps+1})
  +2^{-i+1}(2^{-k_\eps}+2^{-k_\eps-1}).
\]
Now we inductively apply Lemma~\ref{lem:paste} with $U_1=\bigcup_{k=k_\eps}^{\ell-1} V_k$
and $U_2=V_{\ell}$, and $u=w_{i,\ell-1}$ and $v=u_{\ell,i}$, with $\ell>k_{\eps}+1$,
to obtain a
sequence of functions $(w_{i,\ell})_\ell$ with $w_{i,\ell}\in \liploc\left(\bigcup_{k=k_\eps}^{\ell} V_k\right)$ such that $w_{i,\ell}$ vanishes in
$\bigcup_{k=k_\eps}^{\ell}V_k\setminus\Om$,
\[
\int_{\bigcup_{k=k_\eps}^{\ell}V_k}|w_{i,\ell}-\ch_\Om|\, d\mu<2^{-i}\sum_{k=k_\eps}^\infty 2^{-k},
\]
and 
\[
\int_{\bigcup_{k=k_\eps}^{\ell}V_k} g_{w_{i,\ell}}\,d\mu \le \sum_{k=k_\eps}^\infty P_+(\Om,V_k)+2^{-i+1}\sum_{k=k_\eps}^\infty 2^{-k}.
\]
Note in addition that $w_{i,\ell}=w_{i,n+1}$ in $V_n$ for $\ell\ge n+1$. It follows that $w_i:=\lim_{\ell\to\infty}w_{i,\ell}$
exists, belongs to $\liploc\left(\bigcup_{k=k_\eps}^\infty V_k\right)$, and vanishes in
$\bigcup_{k=k_\eps}^\infty V_k\setminus\Om$. Moreover,
\[
\int_{\bigcup_{k=k_\eps}^\infty V_k}|w_i-\ch_\Om|\, d\mu\le 2^{-(i+k_\eps-1)},
\]
and 
\[
\int_{\bigcup_{k=k_\eps}^{\infty}V_k} g_{w_{i}}\,d\mu \le \sum_{k=k_\eps}^\infty P_+(\Om,V_k)+2^{-(i+k_\eps-2)}.
\]
From the first of the above two inequalities it follows that $w_i-\ch_\Om\to 0$ in 
$L^1\left(\bigcup_{k=k_\eps}^\infty V_k\right)$, and so by the second of the above two inequalities,
\[
 P_+\left(\Om, \bigcup_{k=k_\eps}^\infty V_k\right)
 \le\liminf_{i\to\infty}\int_{\bigcup_{k=k_\eps}^{\infty}V_k} g_{w_{i}}\,d\mu
 \le \sum_{k=k_\eps}^\infty P_+(\Om,V_k)<\eps
\]
by~\eqref{eq:tail-end}, as desired. Moreover, $\bigcup_{k=k_\eps}^\infty V_k=W\setminus \overline{W}_{\delta}$
for $\delta:=\delta_{2k_\eps-3}$.
\end{proof}

By combining Lemma~\ref{lem:thin annulus has small measure} and Lemma~\ref{lem:pasting with annulus}, we obtain property (3) of Theorem~\ref{thm:DeGiorgiLetta}, that is,
for any open set $U\subset X$,
\begin{equation}\label{eq:inner regularity}
P_+(\Omega,U)=\sup\{P_+(\Omega,V),\,V\textrm{ open, } V\Subset U\}.
\end{equation}

Finally, we prove property (1) of Theorem~\ref{thm:DeGiorgiLetta}.
\begin{lemma}
Let $U_1,U_2\subset X$ be open sets. Then
\[
P_+(\Omega,U_1\cup U_2)\le P_+(\Omega,U_1)+P_+(\Omega,U_2).
\]
\end{lemma}

\begin{proof}
Take $V\Subset U_1\cup U_2$ and note that $V=U_1'\cup U_2'$ for some $U_1'\Subset U_1$
and $U_2'\Subset U_2$. We can find sequences $(u_i)\subset \liploc(U_1)$ and $(v_i)\subset \liploc(U_2)$ such that
$u_i$ vanishes in $U_1\setminus\Om$, $v_i$ vanishes in
$U_2\setminus\Om$,
$u_i-\ch_\Om\to 0$ in $L^1(U_1)$, $v_i-\ch_\Om\to 0$ in $L^1(U_2)$, and
\[
\lim_{i\to\infty}\int_{U_1}g_{u_i}\,d\mu=P_+(\Om,U_1),\qquad \
\lim_{i\to\infty}\int_{U_2}g_{v_i}\,d\mu= P_+(\Om,U_2).
\]
By Lemma~\ref{lem:paste2}, we then find functions $w_i\in \liploc(U_1'\cup U_2')$ satisfying
$w_i\to\ch_{\Omega}$ in $L^1(U_1'\cup U_2')$ and
\[
\int_{U_1'\cup U_2'}g_{w_i}\,d\mu\le \int_{U_1} g_{u_i}\,d\mu+\int_{U_2}g_{v_i}\,d\mu+C(U_1',U_1)\int_{U_1\cap U_2} |u_i-v_i|\,d\mu.
\]
Note also that by the construction of $w_i$ in Lemma~\ref{lem:paste2}, $w_i$
vanishes in $U_1'\cup U_2'\setminus\Om$.
Letting $i\to\infty$, we obtain
\[
P_+(\Omega,U_1'\cup U_2')\le \liminf_{i\to\infty}\int_{U_1'\cup U_2'}g_{w_i}\,d\mu=
P_+(\Om,U_1)+P_+(\Om,U_2).
\]
By~\eqref{eq:inner regularity}, we obtain the desired conclusion.
\end{proof}

Thus we have proved that $P_+(\Om,\cdot)$ satisfies the conditions of Theorem~\ref{thm:DeGiorgiLetta}, so that $P_+(\Omega,\cdot)$ is a Borel outer measure. Borel regularity follows easily from the definition
\[
P_+(\Omega,A):=\inf\{P_+(\Omega,U):\, U\textrm{ open, }U\supset A\},\quad A\subset X.
\]
In conclusion, $P_+(\Omega,\cdot)$ is a Radon measure.

\noindent Address:\\

\noindent R.K.: Department of Mathematics and Systems Analysis, Aalto University,
P.O. Box 11100, FI-00076 Aalto, Finland. \\
\noindent E-mail: {\tt riikka.korte@aalto.fi}

\vskip .5cm

\noindent P.L.: Department of Mathematical Sciences, P.O. Box 210025, University of
Cincinnati, Cincinnati, OH 45221--0025, U.S.A. \\
\noindent E-mail: {\tt lahtipk@ucmail.uc.edu}

\vskip .5cm

\noindent  X.L.: Department of Mathematics, Sun Yat-sen University, Guangzhou 510275, China.\\
\noindent E-mail: {\tt lixining3@mail.sysu.edu.cn}

\vskip .5cm

\noindent N.S.: Department of Mathematical Sciences, P.O. Box 210025, University of
Cincinnati, Cincinnati, OH 45221--0025, U.S.A. \\
\noindent E-mail: {\tt shanmun@uc.edu} 

\end{document}